\documentclass[11pt]{article}
\usepackage{amsfonts} 
\usepackage{amsmath,hhline,latexsym}
\usepackage[usenames,dvipsnames]{pstricks}
 \usepackage{epsfig}
 \usepackage{pst-grad} % For gradients
 \usepackage{pst-plot} % For axes
\usepackage[utf8]{inputenc}
\usepackage{pgfplots}
 
\usepackage{amssymb}
\usepackage{mathrsfs}
\usepackage[mathscr]{eucal}
\usepackage{graphicx}
\usepackage[francais,english]{babel}
\usepackage{hyperref}
\usepackage{color}
\usepackage{calrsfs}  %%% LETRAS CALIGRÁFICAS
\usepackage{bm}       %%% LETRAS GRIEGAS EN NEGRITA

 \usepackage[usenames,dvipsnames]{pstricks}
 \usepackage{epsfig}
 \usepackage{pst-grad} % For gradients
 \usepackage{pst-plot} % For axes

\newtheorem{thm}{Theorem}
\newtheorem{propo}{Proposition}

\newtheorem{rmq}{Remark}
\newtheorem{lemma}{Lemma}

\newenvironment{proof}[1][Proof]{\noindent\textbf{#1:} }{\Fin}

\def\dis{\displaystyle}
\def\Om{\Omega}
\def\Ombar{\overline{\Omega}}
\def\om{\omega}
\def\omvec{\bm{\omega}}
\def\nvec{\mathbf{n}}
\def\yvec{\mathbf{y}}
\def\xvec{\mathbf{x}}

\def\vvec{\mathbf{v}}
\def\0vec{\mathbf{0}}
\def\zvec{\mathbf{z}}
\def\vvec{\mathbf{v}}

\def\Cvec{\mathbf{C}}
\def\Mvec{\mathbf{M}}

\def\Avec{\mathbf{A}}

\def\Yvec{\mathbf{Y}}
\def\Rvec{\mathbf{R}}
\def\Svec{\mathbf{S}}
\def\Supp{\hbox{\rm Supp\,}}
\def\Zvec{\mathbf{Z}}
\def\fvec{\mathbf{f}}
\def\wvec{\mathbf{w}}
\def\Wvec{\mathbf{W}}

\def\uvec{\mathbf{u}}
\def\gvec{\mathbf{g}}
\def\avec{\mathbf{a}}

\newcommand{\Fin}{\hfill$\Box$}

\newcommand{\N}{\mbox{$I \kern -4pt N$}}

\newcommand{\Q}{\mbox{$Q \kern -8pt I$}}
\newcommand{\R}{\mbox{$I \kern -4pt R$}}
\newcommand{\C}{\mbox{$C \kern -8pt I$}}

\newcommand{\supp}{\operatorname{supp}}

\def\R{{\bf R}}
\begin{document}

\title{\bf{Boundary controllability of incompressible Euler fluids with Boussinesq heat effects}}

	 \author{Enrique \textsc{Fern\'andez-Cara}
	 \thanks{
	 Dpto.\ EDAN, University of Sevilla, Aptdo.~1160, 41080~Sevilla, Spain. E-mail: {\tt cara@us.es}.
	 Partially supported by grant~MTM2010-15592 (DGI-MICINN, Spain).
	 },\ \
	 Maur\'icio C. \textsc{Santos} 
	 \thanks{
	 Departamento de Matem\'{a}tica, Universidade Federal da Para\'iba, 58051-900,
	 Jo\~{a}o Pessoa--PB, Brasil, E-mail: {\tt mcardoso.pi@gmail.com}. Partially supported by CAPES
	 },\ \
	 Diego A. \textsc{Souza}
	 \thanks{
	 Dpto.\ EDAN, University of Sevilla, 41080~Sevilla, Spain and
	 Departamento de Matem\'{a}tica, Universidade Federal da Para\'iba, 58051-900, Jo\~{a}o Pessoa--PB,
	 Brasil. E-mail: {\tt desouza@us.es}. Partially supported by CAPES (Brazil) and grant~MTM2010-15592
	 (DGI-MICINN, Spain).
	 }}
	 
\date{}
	 
\maketitle

\begin{abstract}
	This paper deals with the boundary controllability of inviscid incompressible fluids for which thermal effects are important.
	They will be modeled through the so called Boussinesq approximation. In the zero heat diffusion case, by adapting and 
	extending some ideas from J.-M.~Coron and O.~Glass, we establish 
	the simultaneous global exact controllability of the velocity field and the temperature for 2D and 3D flows.
	When the heat diffusion coefficient is positive, we present some additional results concerning exact controllability for the velocity field 
	and local null controllability of the temperature.
\end{abstract}

\vskip 0.25cm

\noindent
\textbf{Mathematics Subject Classification (2010)-}  93B05, 35Q30, 76C99, 93C20

%%%%%%%%%%%%%%%%%%%%%%%%%%%%%%%%%%%%%%%%%%
%%%% SECTION 1
%%%%%%%%%%%%%%%%%%%%%%%%%%%%%%%%%%%%%%%%%%

\section{Introduction}\label{Sec1}
	Let $\Om \subset \mathbb{R}^N$ be a nonempty, bounded and connected open set whose boundary $\Gamma := \partial\Om$ is of 
	class~$C^{\infty}$, with $N=2$ or $N=3$.	Let $\Gamma_0\subset \Gamma$ be a (small) nonempty open set and let us assume
	that $T>0$. For simplicity, we assume that $\Om$ is simply connected.
	
	In the sequel, we will denote by $C$ a generic positive constant; spaces of $\mathbb{R}^N$-valued functions, as well as their 
	elements, are represented by boldfaced letters; we will denote by $\nvec=\nvec(\xvec)$ the outward unit normal to 
	$\Om$ at points $\xvec\in\Gamma$.
	
	In this work, we will be concerned with the boundary controllability of the system:
\begin{equation}\label{inv-bous}
	\left\{
		\begin{array}{lcl}
			\yvec_t + (\yvec \cdot \nabla) \yvec
			= -\nabla p + \mathbf{k} \, \theta  							&\hbox{ in }& \Om\times(0,T),	\\
			 \nabla \cdot \yvec = 0  										&\hbox{ in }& \Om\times(0,T),	\\
			 \theta_t + \yvec \cdot \nabla \theta = \kappa\, \Delta \theta  			&\hbox{ in }& \Om\times(0,T),	\\
			 \yvec \cdot \nvec = 0	  									&\hbox{ on }& (\Gamma\backslash\Gamma_0)\times(0,T),	\\
			 \yvec(\xvec,0) = \yvec_0(\xvec), \ \ \theta(\xvec,0) = \theta_0(\xvec)   		&\hbox{ in } & \Om.		\\
		\end{array}
	\right.
\end{equation}

	This system models the behavior of an incompressible homogeneous inviscid fluid with thermal effects.
	More precisely,
\begin{itemize}
	\item The field $\yvec$ and the scalar function $p$ stand for the velocity and the pressure of the fluid in~$\Om\times(0,T)$, respectively.
	
	\item	 The function $\theta$ provides the temperature distribution of the fluid.
	
	\item The right hand side $\mathbf{k} \, \theta$ can be viewed as the {\it buoyancy force} density 
		($\mathbf{k} \in \mathbb{R}^N$ is a non-zero vector).
	
	\item The nonnegative constant $\kappa \geq0$ is the heat diffusion coefficient.
\end{itemize}

	This system is relevant for the study and description of atmospheric and oceanographic turbulence, as well as other fluid problems 
	where rotation and stratification play dominant roles (see e.g.~\cite{Pedlo}). In fluid mechanics, \eqref{inv-bous} is used to deal with 
	buoyancy-driven flow; it describes the motion of an incompressible inviscid fluid subject to convected heat transfer under the 
	influence of gravitational forces, see~\cite{Majda-Andrew}.
		
	We will consider the cases $\kappa = 0$ and $\kappa > 0$. When $\kappa=0$,  \eqref{inv-bous} is called the 
	{\it incompressible inviscid Boussinesq} system.
	
	From now on, we assume that $\alpha\in (0,1)$ and we set
\begin{equation}\label{C-space}
\begin{alignedat}{2}
	&\Cvec^{m,\alpha}_0(\Ombar ;\mathbb{R}^N):=\{\, \uvec\in \Cvec^{m,\alpha}(\Ombar ;\mathbb{R}^N)\,:\,  \nabla \cdot \uvec =0 \ \hbox{ in } \ 
	\Ombar, \ \ \uvec\cdot \nvec =0  \ \hbox{ on } \ \Gamma \,\},\\
	&\mathbf{C}(m,\alpha,\Gamma_0):=\{\, \uvec\in \Cvec^{m,\alpha}(\Ombar ;\mathbb{R}^N)\,:\,  \nabla \cdot \uvec =0 \ \hbox{ in } \ \Ombar , \ \ 
	\uvec\cdot \nvec =0  \ \hbox{ on } \ \Gamma\backslash\Gamma_0 \,\},
\end{alignedat}
\end{equation}
	where $ \Cvec^{m,\alpha}(\Ombar ;\mathbb{R}^N)$ denotes the space of $\mathbb{R}^N$-valued functions whose $m$-th 
	order derivatives are {\it H\"older-continuous} in~$\Ombar$ with exponent $\alpha$.
	The usual norms in the Banach spaces $\Cvec^0(\Ombar ;\mathbb{R}^\ell)$ and~$\Cvec^{m,\alpha}(\Ombar ;\mathbb{R}^\ell)$ 
	will be respectively denoted by $\| \cdot \|_0$ and~$\| \cdot \|_{m,\alpha}$.
	We will also need to work with the Banach spaces $C^0([0,T];\Cvec^{m,\alpha}(\Ombar;\mathbb{R}^\ell))$, where the usual norms are
\[
	\| \wvec \|_{0,m,\alpha} := \max_{[0,T]} \| \wvec(\cdot\,,t) \|_{m,\alpha}.
\]
	In particular, $\|\cdot\|_{(0)}$ will stand for $\|\cdot\|_{0,0,0}$.
	
	When $\kappa=0$, it is appropriate to consider the exact boundary controllability problem for \eqref{inv-bous}.
	In general terms, it can be stated as follows:	
%%%% ATTENTION: BOUNDARY CONDITIONS FOR INITIAL AND FINAL TEMPERATURES? NO 
%%
%%

\begin{quote}{\it
	Given $\yvec_0$, $\yvec_1$, $\theta_0$ and~$\theta_1$ in appropriate spaces with $\yvec_0\cdot \nvec=\yvec_1\cdot \nvec=0$ on~$\Gamma\backslash\Gamma_0$, find $(\yvec,p,\theta)$ such that \eqref{inv-bous} holds and, furthermore,
\begin{equation}\label{exact_condition}
	\yvec(\xvec,T) = \yvec_1(\xvec),~~\theta(\xvec,T) = \theta_1(\xvec) \ \hbox{ in } \ \Om.
\end{equation}
}
\end{quote}

	If it is always possible to find $\yvec$, $p$ and $\theta$, it will be said that the incompressible inviscid Boussinesq system is 
	{\it exactly controllable} for $(\Om,\Gamma_0)$ at time $T$.
		
	Notice that, when $\kappa = 0$, in order to determine without ambiguity a unique local in time regular solution to~\eqref{inv-bous}, 
	it is sufficient to prescribe the normal component of the velocity on the boundary of the flow region and, for instance, the full field $\yvec$ 
	and the temperature $\theta$ on the inflow section, i.e.~only where $\yvec\cdot\nvec < 0$, see for instance~\cite{Kazhikhov-1, Kazhikhov-2}.
	Hence, in this case, we can assume that the controls are given as follows:
\begin{equation}\label{3a}
\left\{
\begin{array}{ll}
	\dis \yvec\cdot \nvec\hbox{ on }\Gamma_0\times(0,T),\hbox{ with }\dis\int_{\Gamma_0}\yvec\cdot \nvec \,d\Gamma=0; \\
	\dis \yvec \hbox{ and } \theta \hbox{ at any point of }\Gamma_0\times (0,T) \hbox{ satisfying }\yvec\cdot \nvec<0. \\
	\end{array}
\right.
\end{equation}

	Other choices are possible.
	In any case, once we find a trajectory satisfying~\eqref{inv-bous}  and~\eqref{exact_condition}, there always exists good boundary conditions that furnish controls that drive the state $(\mathbf{y},\theta)$ exactly to the desired target $(\mathbf{y}_1,\theta_1)$.

	The meaning of the exact controllability property is that, when it holds, we can drive the fluid from any initial state $(\yvec_0,\theta_0)$ 
	exactly to any final state $(\yvec_1,\theta_1)$, acting only on an arbitrary small part $\Gamma_0$ of the boundary during an arbitrary 
	small time interval $(0,T)$.
	
	When $\kappa>0$, the situation is different. Due to the {\it regularization effect} of the temperature equation, we cannot expect 
	exact controllability, at least for the temperature.
	
	In order to present a suitable boundary controllability problem, let us introduce another nonempty open set $\gamma \subset \Gamma$.
	Then, the problem is the following:
\begin{quote}{\it
	Given $\yvec_0$, $\yvec_1$ and~$\theta_0$ in appropriate spaces with $\yvec_0\cdot \nvec=\yvec_1\cdot \nvec=0$ on
	~$\Gamma\backslash\Gamma_0$ and~$\theta_0=0$ on $\Gamma\backslash\gamma$, find $(\yvec,p,\theta)$ with $\theta=0$ on 
	$(\Gamma\backslash\gamma)\times(0,T)$ such that \eqref{inv-bous} holds and, furthermore,
\begin{equation}\label{null_exact_condition}
	\yvec(\xvec,T) = \yvec_1(\xvec),~~\theta(\xvec,T) = 0 \ \hbox{ in } \ \Om.
\end{equation}
		}
\end{quote}

	If it is always possible to find $\yvec$, $p$ and $\theta$, it will be said that the incompressible, heat diffusive, inviscid Boussinesq system \eqref{inv-bous} is {\it exactly-null controllable} for $(\Om,\Gamma_0,\gamma)$ at time~$T$.
	
	Note that, if $\kappa > 0$ and we fix the boundary data in~\eqref{3a} for $\yvec$ and (for example) Dirichlet data for $\theta$ of the form
	$$
\theta = \theta_* 1_{\gamma} \ \text{ on } \ \Gamma \times (0,T),
	$$
there exists at most one solution to~\eqref{inv-bous}.
	Therefore, it can be assumed in this case that the controls are the following:
\[
\left\{
\begin{array}{ll}
	\dis \yvec\cdot \nvec\hbox{ on }\Gamma_0\times(0,T),\hbox{ with }\dis\int_{\Gamma_0}\yvec\cdot \nvec \,d\Gamma=0;\\
	\dis \yvec\hbox{ at any point of }\Gamma_0\times (0,T) \hbox{ satisfying }\yvec\cdot \nvec<0;\\
	\dis \theta\hbox{ at any point of }\gamma\times (0,T).
	\end{array}
\right.
\]

	Of course, the meaning of the exact-null controllability property is that, when it holds, we can drive the fluid velocity-temperature pair from any initial state $(\yvec_0,\theta_0)$ exactly to any final state of the form $(\yvec_1,0)$, acting only on arbitrary small parts $\Gamma_0$ and $\gamma$ of the boundary during an arbitrary small time interval $(0,T)$.
	
	In the last decades, a lot of researchers have focused their attention on the controllability of systems governed by (linear and nonlinear) PDEs.	
	Some related results can be found in~\cite{ref1,ref2,ref3,ref4}. In the context of incompressible ideal fluids, this subject has
	been mainly investigated by Coron~\cite{Coron2,Coron3} and~Glass~\cite{Glass1,Glass2,Glass3}.
	
	In this paper, our first task will be to adapt the techniques and arguments of~\cite{Coron3} and~\cite{Glass3} to the situations modeled by~\eqref{inv-bous}.
	Thus, our first main result is the following:
	
\begin{thm}\label{T-INV-BOUS}
	If $\kappa=0$, then the incompressible inviscid Boussinesq system \eqref{inv-bous} is exactly controllable for $(\Om, \Gamma_0)$ 
	at any time~$T>0$. More precisely, for any $\yvec_0, \yvec_1\in \mathbf{C}(2,\alpha,\Gamma_0)$ and  any 
	$\theta_0, \theta_1\in C^{2,\alpha}(\Ombar)$, there exist $\yvec\in C^0([0,T]; \mathbf{C}(1,\alpha,\Gamma_0))$, 
	$\theta\in C^{0}([0,T]; C^{1,\alpha}(\Ombar ))$ and $p\in \mathcal{D'}(\Om\times (0,T))$ such that one has \eqref{inv-bous} 
	and~\eqref{exact_condition}.
\end{thm}

	The proof of Theorem~\ref{T-INV-BOUS} mainly relies on the {\it extension} and {\it return} methods.
		
	These have been applied in several different contexts to establish controllability;
	see the seminal works~\cite{Russell-1} and~\cite{Coron1};
	see also a long list of applications in~\cite{ref1}.

	Let us give a sketch of the strategy used in the proof of Theorem~\ref{T-INV-BOUS}:
	
\begin{itemize}
	\item First, we construct a ``good" trajectory connecting  $(\0vec,0)$ to $(\0vec,0)$ (see Sections~\ref{sec-2-traj} and~\ref{sec-3-traj}).
	
	\item Then, we apply the extension method of David L. Russell~\cite{Russell-1};
	see also~\cite{Littman, Russell-2}.
	It is known that this method turns out to be very useful for a lot of hyperbolic
	(linear and nonlinear) PDEs.
	
	\item Then, we use a {\it Fixed-Point Theorem} and we deduce a local exact controllability result.
	
	\item Finally, we use time scale-invariance and the reversibility in time to obtain the desired global result.
\end{itemize}

	In fact, Theorem~\ref{T-INV-BOUS} is a consequence of the following local result:
	
\begin{propo}\label{P-INV-BOUS}
	Let us assume that $\kappa=0$.
	There exists $\delta>0$ such that, for any $\yvec_0\in \Cvec(2,\alpha,\Gamma_0)$ and any $\theta_0\in C^{2,\alpha}(\Ombar)$ with
$$
\max\left\{\|\yvec_0\|_{2,\alpha},\|\theta_0\|_{2,\alpha}\right\} \leq \delta,
$$
	there exist $\yvec\in C^0([0,1];\Cvec(1,\alpha,\Gamma_0))$, $\theta\in C^0([0,1]; C^{1,\alpha}(\Ombar ))$ and 
	$p\in\mathcal{D}'(\Om \times (0,1))$ satisfying \eqref{inv-bous} in~$\Om\times(0,1)$ and the final conditions
\begin{equation}\label{cont-local}
	\yvec (\xvec,1) = \0vec,~~\theta(\xvec,1) = 0~\hbox{ in }~\Om.
\end{equation}
\end{propo}

	It will be seen later that, in our argument, the $C^{2,\alpha}$-regularity of the initial and final data is needed.
	However, we can only ensure the existence of a controlled solution that is $C^{1,\alpha}$ in space.

	Our second main result is the following:
	
\begin{thm}\label{T-HEAT-INV-BOUS}
	Let $\Om$, $\Gamma_0$ and~$\gamma$ be given and let us assume that $\kappa>0$.
	Then \eqref{inv-bous} is locally exactly-null controllable.
	More precisely, for any $T>0$ and any~$\yvec_0, \yvec_1\in \Cvec_0^{2,\alpha}(\Ombar;\mathbb{R}^N)$, there exists $\eta > 0$, 
	depending on $\|\yvec_0\|_{2,\alpha}$, such that, for each $\theta_0\in C^{2,\alpha}(\Ombar )$ with
\[
\theta_0=0 \ \text{ on }\ \Gamma\backslash\gamma, \ \ \|\theta_0\|_{2,\alpha} \leq \eta,
\]
	we can find $\yvec\in C^0([0,T]; \Cvec^{1,\alpha}(\Ombar;\mathbb{R}^N)$, $\theta\in C^{0}([0,T]; C^{1,\alpha}(\Ombar))$ with 
	$\theta=0$ on $(\Gamma\backslash\gamma)\times(0,T)$, and $p\in \mathcal{D'}(\Om\times (0,T))$ satisfying 
	\eqref{inv-bous} and~\eqref{null_exact_condition}.
\end{thm}

	The proof relies on the following strategy.
	
	First, we linearize and control only the temperature $\theta$;
	this leads the system to a state of the form $(\tilde\yvec_0,0)$ at (say) time $T/2$.
	Then, in a second step, we control the velocity field using in part Theorem~\ref{T-INV-BOUS}.
	It will be seen that, in order to get good estimates and prove the existence of a fixed point, the initial temperature $\theta_0$ must be small.

	To our knowledge, it is unknown whether a global exact-null controllability result holds for~\eqref{inv-bous} when $\kappa > 0$.
	Unfortunately, the cost of controlling $\theta$ grows exponentially with the $L^\infty$-norm of the transporting velocity field $\yvec$ and 
	this is a crucial difficulty to establish estimates independent of the size of the initial data.

\

	The rest of this paper is organized as follows.
	
	In~Section~\ref{Sec2}, we recall the results needed to prove Theorems~\ref{P-INV-BOUS} and~\ref{T-HEAT-INV-BOUS}.
	In~Section~\ref{Sec3}, we give the proof of Theorem~\ref{T-INV-BOUS}.
	In~Section~\ref{Sec4}, we prove Proposition~\ref{P-INV-BOUS} in the 2D case.
	As mentioned above, the main ingredients of the proof are the construction of a nontrivial trajectory that starts and ends at $(\0vec,0)$ 
	and a Fixed-Point Theorem (the key ideas of the return method).
	In~Section~\ref{Sec5}, we give the proof of Theorem~\ref{P-INV-BOUS} in the 3D case.
	Finally, Section~\ref{Sec6} contains the proof of Theorem~\ref{T-HEAT-INV-BOUS}.

%%%%%%%%%%%%%%%%%%%%%%%%%%%%%%%%%%%%%%%%%%
%%%% SECTION 2
%%%%%%%%%%%%%%%%%%%%%%%%%%%%%%%%%%%%%%%%%%

\section{Preliminary results}\label{Sec2}

	In this section, we will recall some results used in the proofs of Theorems~\ref{T-INV-BOUS} and~\ref{T-HEAT-INV-BOUS}.
	Also, we will indicate how to construct a trajectory appropriate to apply the return method.

	The following result is an immediate consequence of {\it Banach's Fixed-Point Theorem:}
	
\begin{thm}\label{fixedpoint}
	Let $(B_1,\|\cdot\|_1)$ and $(B_2,\|\cdot\|_2)$ be Banach spaces with $B_2$ continuously embedded in $B_1$.
	Let $B$ be a subset of $B_2$ and let $G: B \mapsto B$ be a  uniformly continuous mapping such that, for some $n \geq 1$ and  some $\alpha \in [0,1)$, one has
$$
	\|G^n(u)-G^n(v)\|_1\leq \alpha \|u-v\|_1 \ \ \forall u,v\in B.
$$
	Let us denote by $\overline B$ the closure of $B$ for the norm $\|\cdot\|_{1}$.
	Then, $G$ can be uniquely extended to a continuous mapping $\widetilde G: \overline B \mapsto \overline B$ that 
	possesses a unique fixed-point in $\overline B$.
\end{thm}

	Recall that, if $E$ is a Banach space with norm $\|\cdot\|_E$ and~$f: [0,T] \mapsto E$ is continuously differentiable, then $t \mapsto \| f(t) \|_E$ is {\it right-differentiable,} with
	\[
{d \over dt^+} \| f(t) \|_E \leq \| f'(t) \|_E
	\]
for all $t$.
	Later, the following lemma will be very important to deduce appropriate estimates;
	this is Lemma~1, p.~6, in~\cite{Bardos}.
	
	%The proof can be found in \cite{Bardos}.

\begin{lemma}\label{lemma Bardos1}
	Let $m$ be a nonnegative integer.
	Assume that $u\in C^0([0,T];C^{m+1,\alpha}(\overline{\Omega}))$, $g\in C^0([0,T];C^{m,\alpha}(\overline{\Omega}))$ and $\vvec\in C^0([0,T];\Cvec^{m,\alpha}(\overline{\Omega};\mathbb{R}^N))$ are given, with $\vvec\cdot \nvec=0$ on~$\Gamma \times (0,T)$ and
	\[
\frac{\partial u}{\partial t}+\vvec\cdot\nabla u= g \ \hbox{ in } \ \Omega\times(0,T).
	\]
	Then, $u_t \in C^0([0,T];C^{m,\alpha}(\overline{\Omega}))$.
	If $m\geq1$, one has
\[
	\frac{d}{d t^+} \, \|u(\cdot\,,t)\|_{m,\alpha}\leq \|g(\cdot\,,t)\|_{m,\alpha}+K\|\vvec(\cdot\,,t)\|_{m,\alpha}\|u(\cdot\,,t)\|_{m,\alpha}
	\quad \hbox{in} \quad  (0,T),
\]
	where $K$ is a constant only depending on $\alpha$ and $m$.
	On the other hand, if $m=0$, the following holds:
\[
    \frac{d}{d t^+} \, \|u(\cdot\,,t)\|_{0,\alpha}\leq \|g(\cdot\,,t)\|_{0,\alpha}+\alpha\| \nabla\vvec(\cdot\,,t) \|_{0,\alpha} \| u(\cdot\,,t) \|_{0,\alpha}\quad \hbox{in} \quad  (0,T).
\]
\end{lemma}

	From Lemma~\ref{lemma Bardos1} and a standard regularization argument, we easily deduce the following:
	
\begin{lemma}\label{lemma Bardos}
	Let $m$ be a nonnegative integer.
	Assume that $u\in C^0([0,T];C^{m,\alpha}(\overline{\Omega}))$, $g\in C^0([0,T];C^{m,\alpha}(\overline{\Omega}))$ and $\vvec\in C^0([0,T];\Cvec^{m,\alpha}(\overline{\Omega};\mathbb{R}^N))$ are given, with $\vvec\cdot \nvec=0$ on~$\Gamma \times (0,T)$ and
	\[
\frac{\partial u}{\partial t}+\vvec\cdot\nabla u= g \ \hbox{ in } \ \Omega\times(0,T).
	\]
	Then
\[
	\dis\|u\|_{0,m,\alpha}\leq\left(\int_0^T \|g(\cdot\,,t)\|_{m,\alpha}\,dt+\|u(\cdot\,,0)\|_{m,\alpha}\right)\exp\left(\dis K\int_0^T\|\vvec(\cdot\,,t)\|_{m,\alpha}\,dt\right),
\]
	where $K$ is a constant only depending on $\alpha$ and $m$.
\end{lemma}

	We will also use he following technical result (see~Lemma $3.1$, p. 8, in~\cite{Glass3}):
	
\begin{lemma}\label{lemma div-om}
	Let us assume that
\[
	\begin{array}{cll} \dis
	\wvec_0\in \Cvec^{1,\alpha}(\bar\Om;\mathbb{R}^N), \ \ \nabla\cdot \wvec_0 = 0 &  \hbox{in} & \Om, \\
	\uvec\in C^0([0,T];\Cvec^{1,\alpha}(\bar\Om;\mathbb{R}^N)),\ \  \uvec\cdot \nvec=0&  \hbox{on} & \Gamma \times (0,T), \\
	\gvec\in C^0([0,T];\Cvec^{0,\alpha}(\bar\Om,\mathbb{R}^N)),\ \  \nabla\cdot \gvec = 0&  \hbox{in} & \Om\times(0,T).
	\end{array}
\]
	Let~$\wvec$ be a function in~$C^0([0,T];\Cvec^{1,\alpha}(\overline\Om;\mathbb{R}^N))$ satisfying
\[
	\left\{
	\begin{array}{lll}
	\wvec_t+(\uvec\cdot\nabla)\wvec=(\wvec\cdot\nabla)\uvec-(\nabla\cdot \uvec)\wvec + \gvec   & \hbox{ in} & \Omega\times (0,T), \\
	\noalign{\smallskip} \dis \wvec(\cdot\,,0) = \wvec_0                         & \hbox{ in } & \Om.                \\
\end{array}
\right.
\]
	Then, $\nabla\cdot \wvec\equiv0$.
	Moreover, there exists $\vvec\in C^0([0,T];\Cvec^{2,\alpha}(\overline\Om;\mathbb{R}^N))$ such that
\[
	\wvec=\nabla\times \vvec~\hbox{ in }~\Om\times(0,T).
\]
\end{lemma}

	To end this section, we will recall a well known result dealing with the null controllability of general parabolic linear systems of the form
\begin{equation}\label{parab}
	\left\{
		\begin{array}{lll}
		     u_t - \kappa\Delta u + \wvec\cdot\nabla u = v1_{\om}        		& \text{in}  	&  D\times(0,T),  \\
		     u = 0               		    								& \text{on} 	&  \partial D\times(0,T),                   \\
		     u(\xvec,0) = u_0(\xvec)                           	  				& \text{in}  	&  D,
		\end{array}
	\right.
\end{equation}
	where $D \subset \mathbb{R}^N$ is a nonempty bounsded open set, $\kappa > 0$, $\wvec\in \mathbb{L}^\infty(D\times(0,T))$, $\om \subset D$ is a nonempty open set and 
	$1_\om$ is the characteristic function of~$\om$.

	It is well known that, for each $u_0 \in L^2(D)$ and each $v \in L^2(\om \times (0,T))$, there exists exactly one solution $u$ to \eqref{parab}, with
\[
	u\in C^0([0, T];L^2(D))\cap L^2(0, T;H^1_0(D)).
\]

	We also have:
\begin{thm}\label{NC-Parab}
	The linear system \eqref{parab} is null-controllable at any time $T>0$.
	In other words, for each $u_0\in L^2(D)$ there exists $v \in L^2(\om\times(0,T))$ such that the associated solution to~\eqref{parab} satisfies
\begin{equation}\label{8p}
	u(\xvec,T)=0 ~~ \text{in} ~~ D.
\end{equation}
	Furthermore, the extremal problem
\begin{equation}\label{optm}
	\left\{
		\begin{array}{l}
			\noalign{\smallskip} \text{Minimize } \ \displaystyle{\frac{1}{2}\iint_{\om\times(0,T)}|v|^2 \,dx\,dt}\\
			\noalign{\smallskip} \text{Subject to:  $v \in L^2(\om\times(0,T))$, $u$ satisfies \eqref{8p}}			
		\end{array}
	\right.
\end{equation}
	possesses exactly one solution $\hat v$ satisfying
\begin{equation}\label{est-cont}
	\|\hat v\|_2\leq C_0\|u_0\|_2,
\end{equation}
	where
\[
	C_0=\exp \left( C_1 \left( 1+ {1 \over T} + (1+T^2)\| \wvec \|^2_\infty \right) \right)
\]
	and $C_1$ only depends on $D$, $\om$ and $\kappa$.
\end{thm}

	For more details, see for instance Theorem~$2.2$, p.~1416, in~\cite{Cara_guerrero}.

%%%%%%%%%%%%%%%%%%%%%%%%%%%%%%%%%%%%%%%%%%
%%%% SUBSECTION 2.1
%%%%%%%%%%%%%%%%%%%%%%%%%%%%%%%%%%%%%%%%%%

\subsection{Construction of a trajectory when $N=2$}\label{sec-2-traj}

	We will argue as in \cite{Coron3}.
	Thus, let $\Omega_1 \subset \mathbb{R}^2$ be a bounded, Lipschitz-contractible open set whose boundary 
	consists of  two disjoint closed line segments $\Gamma^{-}$ and $\Gamma^{+}$ and two disjoint curves $\Sigma'$ and $\Sigma''$ 
	of class $C^{\infty}$ such that $\partial\Sigma'\cup\partial\Sigma''=\partial\Gamma^{-}\cup\partial\Gamma^{+}$.
	
	We assume that $\Om \subset \Om_1$.
	We also impose that there is a neighborhood $U^{-}$ of $\Gamma^{-}$ (resp. $U^+$ of $\Gamma^+$) such that $\Omega_1\cap U^{-}$
	(resp. $\Omega_1\cap U^+$) coincides with the  intersection of $U^{-}$ (resp. $U^+$), an open semi-plane limited by the line containing
	$\Gamma^{-}$ (resp. $\Gamma^{+}$) and the band limited by the two straight lines orthogonal to $\Gamma^{-}$ (resp. $\Gamma^{+}$)
	and passing through $\partial\Gamma^{-}$ (resp. $\partial\Gamma^{+}$);
	see~Fig.~\ref{FIG1}.
	
%%%% FIGURE
\begin{figure}[h]
\vskip-3cm
\centering\includegraphics[scale=0.5]{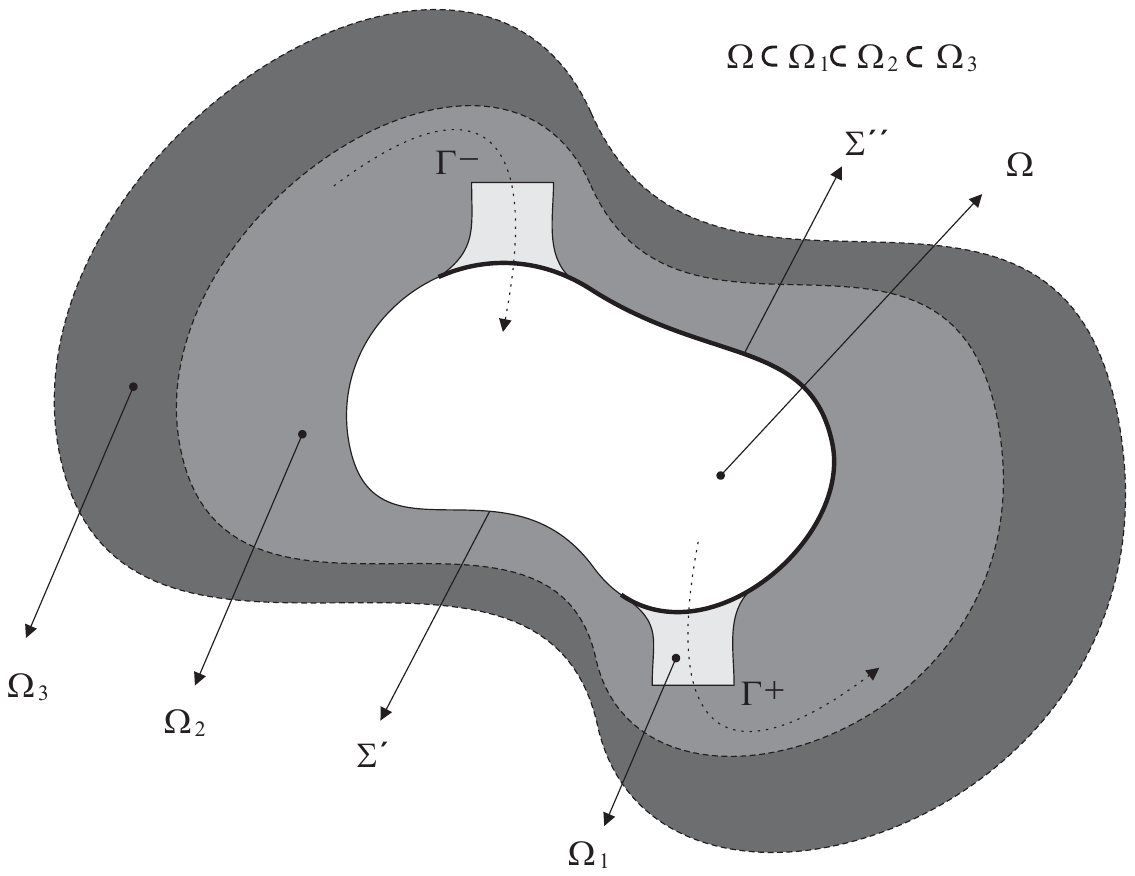}
\vskip-3cm
    \caption{The domains $\Om$, $\Om_1$, $\Om_2$ and~$\Om_3$}\label{FIG1}
\end{figure}
	
	Let $\varphi$ be the solution to	
\begin{equation}\label{trajectory}
	\left\{
		\begin{array}{lll}
    			-\Delta \varphi=0										& \text{in}& \Om_1,          \\
			\noalign{\smallskip} \dis \varphi = 1                                  			& \text{on}& \Gamma^+,  \\
			\noalign{\smallskip} \dis \varphi = -1                                 			& \text{on}& \Gamma^{-}, \\
			\noalign{\smallskip} \dis \frac{\partial\varphi}{\partial n} = 0 	& \text{on}& \Sigma,
		\end{array}
	\right.
\end{equation}
	where $\Sigma=\Sigma'\cup\Sigma''$.
	Then, we have the following result from J.-M. Coron \cite{Coron2}, p.~273:
	
\begin{lemma}\label{L0}
	One has $\varphi\in C^{\infty}(\Ombar_1)$, $-1<\varphi(\xvec)<1$ for all $\xvec\in\Omega_1$ and
\begin{equation}\label{phi}
	\nabla\varphi (\xvec) \neq \0vec\quad\forall\xvec\in \Ombar _1.
\end{equation}
\end{lemma}

	Let $\gamma\in C^{\infty}([0,1])$  be a non-zero function such that $\text{Supp\,} \gamma \subset(0,1/2)\cup (1/2,1)$ and the sets $(\text{Supp\,} \gamma)\cap (0,1/2)$ and $(\text{Supp\,} \gamma)\cap (1/2,1)$ are non-empty.
	
	Let $M>0$ be a constant to be chosen below and set
	\[
\overline{\yvec}(\xvec,t):=M\gamma(t)\nabla\varphi(\xvec), \ \
\overline p(\xvec,t):=-M\gamma_t(t)\varphi(\xvec)-\frac{M^2}{2}\gamma(t)^2|\nabla\varphi(\xvec)|^2, \ \ 
\overline\theta\equiv0.
	\]
	Then \eqref{inv-bous} is satisfied by $(\overline\yvec,\overline p,\overline\theta)$ for $T=1$, $\yvec_0=\0vec$ and~$\theta_0=0$.
	The triplet $(\overline\yvec,\overline p,\overline\theta)$ is thus a nontrivial trajectory of~\eqref{inv-bous} that connects the zero state to itself.

	Let $\Om_3$ be a bounded open set of class $C^{\infty}$ such that $\Om_1\subset\subset \Om_3$.
	We extend $\varphi$ to $\Ombar_3$ as a $C^{\infty}$ function with compact support in $\Om_3$ and we still denote this extension by $\varphi$.
	Let us introduce $\yvec^*(\xvec,t):=M\gamma(t)\nabla \varphi(\xvec)$ (observe that $\overline\yvec$ is the restriction of $\yvec^*$ to 
	$\Ombar\times [0,1]$).
	Also, consider the associated {\it flux function} ~$\Yvec^*:\Ombar_3\times[0,1]\times[0,1]\mapsto \Ombar_3$, defined as follows:
\begin{equation}\label{Xestrella}
	\left\{
		\begin{array}{l}
			\Yvec^*_t(\xvec,t,s) = \yvec^*(\Yvec^*(\xvec,t,s),t)\\
			\Yvec^*(\xvec,s,s)=\xvec.
		\end{array}
	\right.
\end{equation}
	Obviously, $\Yvec^*$ contains all the information on the trajectories of the particles transported by the velocity field $\yvec^*$.
	The flux $\Yvec^*$ is of class $C^{\infty}$ in $\Ombar_3\times[0,1]\times[0,1]$.
	Furthermore, $\Yvec^*(\cdot\,,t,s)$ is a diffeomorphism of $\Ombar_3$ onto itself
	and $(\Yvec^*(\cdot\,,t,s))^{-1}=\Yvec^*(\cdot\,,s,t)$ for all $s,t\in [0,1]$.	
\begin{rmq}\label{IN-OUT}{\rm
	From the definition of $\yvec^*$ and the boundary conditions on $\Om_1$ satisfied by $\varphi$, we observe that the particles cannot 
	cross $\Sigma$.
	Since $\varphi$ is constant on $\Gamma^+$, the gradient $\nabla\varphi$ is parallel to the normal vector on $\Gamma^+$.
	Since $\varphi$ attains a maximum at any point of $\Gamma^+$,  we have $\nabla\varphi\cdot\nvec > 0$ on $\Gamma^+$, 
	whence $\yvec^* \cdot \nvec \geq 0$ on~$\Gamma^+\times [0,1]$. Similarly, $\yvec^* \cdot \nvec \leq 0$ on $\Gamma^{-}\times[0,1]$.
	Consequently, the particles having velocity $\yvec^*$ can leave $\Omega_1$ only through~$\Gamma^+$ and can enter 
	$\Omega_1$ only through~$\Gamma^{-}$.
}
\Fin
\end{rmq}

	The following lemma shows that the particles that travel with velocity $\yvec^*$ and are inside $\Ombar_1$ at time
	$t=0$ (resp.\,\,$t=1/2$) will be outside $\Ombar_1$ at time $t=1/2$ (resp.~$t=1$).
	
\begin{lemma}\label{Lemma-M}
	There exist $M>0$ (large enough) and a bounded open set $\Om_2$ satisfying $\Om_1 \subset\subset \Om_2 \subset\subset \Om_3$ such that
\begin{equation}\label{OM_2}
	\Yvec^*(\xvec,1/2,  0) \not \in\Ombar_2 \ \ \text{and} \ \ \Yvec^*(\xvec,1,1/2) \not \in\Ombar_2 \ \ \forall \xvec\in \overline\Omega_2.
\end{equation}
\end{lemma}

	The proof is given in \cite{Coron3} and relies on the properties of $\yvec^*$ and, more precisely, on the fact that
	$t\mapsto\varphi(\Yvec^*(\xvec,t,s))$ is nondecreasing.
	
	The next step is to introduce appropriate extension mappings from $\Om$ to $\Om_3$.
	We have the following result from~\cite{Hamil}
	(see Corollary~1.3.7, p.~138):

\begin{lemma}\label{L3}
	For $\ell=1$ and $\ell = 2$, there exist continuous linear mappings $\pi_\ell:\Cvec^0(\overline\Om;\mathbb{R}^\ell)\mapsto 
	\Cvec^0(\Ombar_3;\mathbb{R}^\ell)$	such that
\[
	\left\{
		\begin{array}{l}
			\noalign{\smallskip}
			\dis\pi_\ell(\fvec)=\fvec~~\hbox{in}~~\Om~~\text{and}~~\text{Supp\,} \pi_\ell(\fvec)\subset\Om_2\quad \forall \fvec 
			\in \Cvec^0(\overline\Om;\mathbb{R}^\ell), \\
			\noalign{\smallskip}
			\dis \pi_\ell~\text{maps continuously}~ \Cvec^{m,\lambda}(\overline\Om;\mathbb{R}^\ell)
		 	~\text{into}~\Cvec^{m,\lambda}(\overline{\Om_3};\mathbb{R}^\ell) \quad \forall m\geq 0,\ \ \forall\lambda\in(0,1).
		\end{array}
	\right.
\]
\end{lemma}

	The next lemma asserts that \eqref{OM_2} holds not only for $\yvec^*$ but also for any appropriate extension of any flow $\zvec$ 
	close enough to $\overline\yvec$:
	
\begin{lemma}\label{nu-small}
	For each $\zvec\in C^0(\Ombar\times [0,1];\mathbb{R}^2)$, let us set $\zvec^*=\yvec^*+\pi_2(\zvec-\overline{\yvec})$.
	There exists $\nu>0$ such that, if $\|\zvec-\overline{\yvec}\|_{(0)}\leq \nu$, then
\begin{equation}\label{OM_3}
	\Zvec^*(\xvec,1/2,  0) \not \in\Ombar_2 \ \ \text{and} \ \ \Zvec^*(\xvec,1,1/2) \not \in\Ombar_2 \ \ \forall \xvec\in \overline\Omega_2,
\end{equation}
	where $\Zvec^*$ is the flux function associated to $\zvec^*$.
\end{lemma}

\begin{proof}
	Let us set
$$
	\Avec=\left\{\Yvec^*(\xvec,1/2,0)\,:\,\xvec\in \Ombar_2\right\}\cup\left\{\Yvec^*(\xvec,1,1/2)\,:\, \xvec\in \Ombar_2\right\}.
$$
	Both $\Avec$ and $\Ombar_2$ are compact subsets of $\mathbb{R}^2$ and, in view of Lemma~\ref{Lemma-M}, 
	$\Avec \cap \Ombar_2 = \emptyset$.
	Consequently, $d:=\text{dist\,} (\Avec,~\Ombar_2)>0$.
	
	Let us introduce $\Wvec:=\Yvec^*-\Zvec^*$.
	Then, in view of the {\it Mean Value Theorem} and the properties of $\pi_2$, we have:
\[
\begin{alignedat}{2}
	\noalign{\smallskip}|\Wvec(\xvec,t,s)|\leq
	&~\dis M\int_s^t\gamma(\sigma)|\nabla\varphi(\Yvec^*(\xvec,\sigma,s))-\nabla\varphi(\Zvec^*(\xvec,\sigma,s))|\,d\sigma\\
	\noalign{\smallskip}
	&\dis+\int_s^t|\pi_2(\zvec-\overline\yvec)(\Zvec^*(\xvec,\sigma,s),\sigma)|\,d\sigma\\
	\noalign{\smallskip}\leq
	&~\dis  M\|\nabla\varphi\|_0 \int_s^t\gamma(\sigma)|\Wvec(\xvec,\sigma,s)|\,d\sigma+ \int_s^t\|(\pi_2(\zvec-\overline\yvec))(\cdot,\sigma)\|_0\,
	d\sigma\\
	\noalign{\smallskip}\dis\leq
	&~\dis M\|\nabla\varphi\|_0 \int_s^t\gamma(\sigma)|\Wvec(\xvec,\sigma,s)|\,d\sigma+ C\int_s^t\|(\zvec-\overline\yvec)(\cdot,\sigma)\|_0\,d\sigma,\\
\end{alignedat}
\]	
	where $(\xvec,t,s)\in \Ombar_3\times[0,1]\times[0,1]$.
	Therefore, from {\it Gronwall's Lemma,} we find that
\[
\begin{alignedat}{2}
	\noalign{\smallskip}|\Wvec(\xvec,t,s)|
	\leq&~\dis C\left(\int_s^t\|\zvec-\overline\yvec\|_0(\sigma)\,d\sigma\right)\exp\dis\left( M\|\nabla\varphi\|_0 \int_s^t\gamma(\sigma)\,d\sigma \right)\\
	\noalign{\smallskip}
	\leq&~\dis Ce^{M\|\nabla\varphi\|_0 \|\gamma\|_0}\|\zvec-\overline\yvec\|_{(0)}
\end{alignedat}
\]
	and, consequently, there exists $\nu>0$ such that, if $\|\zvec-\overline\yvec\|_{(0)}\leq \nu$, one has
\begin{equation}\label{d/2}
\begin{array}{lll}
	|\Wvec(\xvec,t,s)|\leq\dis \frac{d}{2} \quad \forall(\xvec,t,s)\in \Ombar_3\times[0,1]\times[0,1].
\end{array}
\end{equation}
	Thanks to Lemma~\ref{Lemma-M} and~\eqref{d/2}, we necessarily have \eqref{OM_3} and the proof is achieved.
\end{proof}

%%%%%%%%%%%%%%%%%%%%%%%%%%%%%%%%%%%%%%%%%%
%%%% SUBSECTION 2.2
%%%%%%%%%%%%%%%%%%%%%%%%%%%%%%%%%%%%%%%%%%

\subsection{Construction of a trajectory when $N=3$}\label{sec-3-traj}

	In this Section, we follow \cite{Glass3}.
	As in the two-dimensional case, $\overline \yvec$ will be of the potential form $``\nabla\varphi"$, with the property that any particle traveling 
	with velocity $\overline{\yvec}$ must leave $\overline \Om$ at an appropriate time.
	The main difference will be that, in this three-dimensional case, $\varphi$ will not be chosen independent of $t$.
%
%
%	Add more information on 3d case. 
%
%

	We first recall a lemma:
\begin{lemma}\label{L1}
	Let $\mathscr{O}$ be a regular bounded open set such that $\Om\subset\subset \mathscr{O} $.
	For each $\avec\in\Ombar$, there exists $\phi^\avec\in C^{\infty}(\overline{\mathscr{O}}\times[0,1])$
	such that ~$\supp(\phi^\avec)\subset \mathscr{O}\times\left(1/4,3/4\right)$,
\begin{equation}\label{beta}
\begin{array}{cll}
\left\{
\begin{array}{lll}
	\noalign{\smallskip} \dis	-\Delta\phi^\avec=0 & \hbox{ in } &\Omega\times(0,1),\\
	\noalign{\smallskip} \dis	{\partial \phi^\avec\over\partial \nvec}=0 & \hbox{ on }& (\Gamma\setminus\Gamma_0)\times(0,1)
\end{array}
\right.
\end{array}
\end{equation}
	and
$$
	\mathbf{\Phi}^{\avec}(\avec,1,0)\in \mathscr{O}\setminus\overline{\Omega},
$$
	where $\mathbf{\Phi}^{\avec}:=\mathbf{\Phi}^{\avec}(\xvec,t,s)$ is the flux associated to $\nabla\phi^\avec$, that is, the unique 
	$\mathbb{R}^N-$valued function in 
	$\overline{\mathscr{O}} \times [0,1] \times [0,1]$ satisfying
\begin{equation}\label{Xa}
\left\{
\begin{array}{lll}
	\noalign{\smallskip} \dis\mathbf{\Phi}^{\avec}_t(\xvec,t,s)=\nabla\phi^\avec(\mathbf{\Phi}^{\avec}(\xvec,t,s),t), \\
	\noalign{\smallskip} \dis\mathbf{\Phi}^{\avec}(\xvec,s,s)=\xvec.
\end{array}
\right.
\end{equation}
\end{lemma}

	The proof is given in~\cite{Glass3}
	(see Lemma~2.1, p.~3).

	With the help of these $\mathbf{\Phi}^\avec$, we can construct a vector field $\yvec^*$ in $\mathscr{O}\times(0,1)$ that makes
	the particles go from $\Omega$ to the outside and then makes them come back.
%
% 
%  arrumar essa frase
%

	Indeed, from the continuity of the functions $\mathbf{\Phi}^\avec$ and the compactness of $\Ombar$, we can find $\avec_1,\avec_2$,
	$\ldots,\avec_k$ in $\Ombar$, real numbers $r_1,\ldots,r_k$, smooth functions $\phi^1:=\phi^{\avec_1},\ldots,\phi^k:=\phi^{\avec_k}$
	satisfying Lemma~\ref{L1} and a bounded open set $\mathscr{O}_0$ with $\Om\subset\subset \mathscr{O}_0\subset\subset \mathscr{O}$, 
	such that
\begin{eqnarray}\label{omegatil}
    	\Ombar\subset\bigcup_{i=1}^{k}B^i\subset\subset \mathscr{O}_0\hbox{~~and~~} 
	\mathbf{\Phi}^i(\overline{B}^i,1,0)\subset \mathscr{O}\setminus\overline{\mathscr{O}}_0,
\end{eqnarray}
	where $B^i:=B(\avec_i;r_i)$ and $\mathbf{\Phi}^i:=\mathbf{\Phi}^{\avec_i}$ for $i=1,\ldots,k$.
	
	As in \cite{Glass3}, the definition of $\yvec^*$ is as follows: let the time $t_i$ be given by
\begin{equation}\label{7}
\begin{array}{l}
	\noalign{\smallskip} \dis t_i=\frac{1}{4}+\frac{i}{4k},~~~ i = 0,\ldots,2k, \\
	\noalign{\smallskip} \dis t_{i+1/2}=\frac{1}{4}+\left(i+{1\over2}\right){1\over4k},~~~ i = 0,\ldots,2k-1
\end{array}
\end{equation}
	and let us set
\begin{equation}\label{betay*}
	\phi(\xvec,t)=
\left\{
\begin{array}{ll}
	\noalign{\smallskip}\dis  0,     					& (\xvec,t)\in{\overline{\mathscr{O}}}\times ([0,1/4] \cup [3/4,1]),		\\	
	\noalign{\smallskip}\dis 8k\phi^{j}(\xvec,8k(t-t_{j-1})),	&(\xvec,t) \in{\overline{\mathscr{O}}}\times\left[t_{j-1},t_{j-1/2}\right], 	\\
	\noalign{\smallskip}\dis-8k\phi^{j}(\xvec,8k(t_j-t)),      	&(\xvec,t) \in{\overline{\mathscr{O}}}\times\left[t_{j-1/2},t_{j}\right]
\end{array}
\right.
\end{equation}
	for $j=1,\ldots,2k$, where $\phi^{k+i}:=\phi^i$ for $i=1,\ldots,k$.
	Then, we set $\yvec^*:=\nabla\phi$ and
	$\overline{\yvec}:=\yvec^*|_{\overline{\Omega}\times[0,1]}$ and we denote by $\Yvec^*$ the flux associated to~$\yvec^*$.

	If we set $\bar p(\xvec,t):=-\phi_t(\xvec,t)-\frac{1}{2}|\nabla \phi(\xvec,t)|^2$ and $\overline\theta\equiv0$, then \eqref{inv-bous} and~\eqref{exact_condition} are verified by $(\overline\yvec,\overline p,\overline\theta)$ for $T=1$, $\yvec_0=\yvec_1=\0vec$ and~$\theta_0=\theta_1=0$.

	\
	
	Thanks to \eqref{omegatil} and \eqref{betay*}, one has:
\begin{lemma}\label{Lemmay*}
	The following property holds for all $i=1,\ldots,k$:
\begin{equation}\label{BBB-ii}
	\Yvec^*(\xvec,t_{i-1/2},0) \in \mathscr{O}\setminus \overline{\mathscr{O}}_0\quad \hbox{and}\quad 
	\Yvec^*(\xvec,t_{k+i-1/2},1/2) \in \mathscr{O}\setminus \overline{\mathscr{O}}_0 \quad \forall \xvec\in B^i.
\end{equation}
\end{lemma}	

	For the proof, it suffices to notice that, in $ \overline{\mathscr{O}}\times [1/4,3/4] \times [1/4,3/4]$, $\Yvec^*(\xvec,t,s)$ is given as
	follows:
\[
			\left\{
				\begin{array}{ll}
					\noalign{\smallskip} \dis
					\mathbf{\Phi}^j(\xvec,8k(t-t_{j-1}),8k(s-t_{l-1}))~~\hbox{if}~~(\xvec,t,s)\in\overline{\mathscr{O}}
					\times[t_{j-1},t_{j-1/2}]\times[t_{l-1},t_{l-1/2}],\\
					\noalign{\smallskip} \dis
					\mathbf{\Phi}^j(\xvec,8k(t-t_{j-1}),8k(t_l-s))~~\hbox{if}~~(\xvec,t,s)\in\overline{\mathscr{O}}
					\times[t_{j-1},t_{j-1/2}]\times[t_{l-1/2},t_l],\\
					\noalign{\smallskip} \dis
					\mathbf{\Phi}^j(\xvec,8k(t_j-t),8k(s-t_{l-1}))~~\hbox{if}~~(\xvec,t,s)\in\overline{\mathscr{O}}
					\times[t_{j-1/2},t_j]\times[t_{l-1},t_{l-1/2}],\\
					\noalign{\smallskip} \dis
					\mathbf{\Phi}^j(\xvec,8k(t_j-t),8k(t_l-s))~~\hbox{if}~~(\xvec,t,s)\in \overline{\mathscr{O}}
					\times[t_{j-1/2},t_j]\times[t_{l-1/2},t_l]
				\end{array}
			\right.
\]
	for all $l,j=1,\ldots,2k$, where $\mathbf{\Phi}^{k+i}$ the flux associated to~$\nabla\phi^{k+i}$ for $i=1,\ldots,k$.
	
	Hence, one has the following for all $i=1,\ldots,k$ and for each $\xvec\in B^i$ :
\[
	\Yvec^{*}(\xvec,t_{i-1/2},0)=\Yvec^{*}(\xvec,t_{i-1/2},1/4)=\Yvec^{*}(\xvec,t_{i-1/2},t_0)
	=\mathbf{\Phi}^{i}(\xvec,1,0)\in  \mathscr{O}\setminus\overline{\mathscr{O}}_0
\]
	and
\[
	 \Yvec^{*}(\xvec,t_{k+i-1/2},1/2)=\Yvec^{*}(\xvec,t_{k+i-1/2},t_k)=\mathbf{\Phi}^{k+i}(\xvec,1,0)=\mathbf{\Phi}^i(\xvec,1,0)
	\in  \mathscr{O}\setminus\overline{\mathscr{O}}_0.
\]

	A result similar to Lemma \ref{L3} also holds here:
\begin{lemma}\label{L3p}
	For $\ell=1$ and $\ell = 3$, there exist continuous linear mappings 
	$\pi_\ell:\Cvec^0(\overline\Om;\mathbb{R}^\ell)\mapsto \Cvec^0(\overline{\mathscr{O}};\mathbb{R}^\ell)$ such that
\[
	\left\{
		\begin{array}{l}
			\noalign{\smallskip}\dis
			\pi_\ell(\fvec)=\fvec~~\hbox{in}~~\Om~~\text{and}~~
			\text{Supp\,} \pi_\ell(\fvec)\subset \mathscr{O}_0\quad \forall \fvec \in \Cvec^0(\overline\Om;\mathbb{R}^\ell), \\
			\noalign{\smallskip}
			\dis \pi_\ell~\text{maps continuously}~ \Cvec^{n,\lambda}(\overline\Om;\mathbb{R}^\ell)
		 	~\text{into}~\Cvec^{n,\lambda}(\overline{\mathscr{O}};\mathbb{R}^\ell) \quad \forall n\geq 0,\ \ \forall\lambda\in(0,1).
		\end{array}
	\right.
\]
\end{lemma}

	Finally, we have that~\eqref{BBB-ii} also holds for the flux corresponding to any velocity field close enough 
	to $\overline{\yvec}$:

\begin{lemma}\label{Lemmalocal}
	For each $\zvec\in C^0(\Ombar\times [0,1];\mathbb{R}^3)$, let us set $\zvec^*=\yvec^*+\pi_3(\zvec-\overline{\yvec})$.
	Then there exists $\nu>0$ such that, if $\|\zvec-\overline{\yvec}\|_{(0)}\leq \nu$ and $ i=1,\ldots,k$, one has:
\begin{equation*}\label{B-near}
	\Zvec^*(\xvec,t_{i-1/2},  0) \in\mathscr{O}\setminus\overline{\mathscr{O}}_0 
	\quad\hbox{and}\quad 
	\Zvec^*(\xvec,t_{k+i-1/2},1/2) \in \mathscr{O}\setminus\overline{\mathscr{O}}_0 \quad \forall \xvec\in B^i,
\end{equation*}
	where $\Zvec^*$ is the flux associated to $\zvec^*$.
\end{lemma}

	The proof is very similar to the proof of Lemma~\ref{nu-small} and will be omitted.

%%%%%%%%%%%%%%%%%%%%%%%%%%%%%%%%%%%%%%%%%%
%%%% SECTION 3
%%%%%%%%%%%%%%%%%%%%%%%%%%%%%%%%%%%%%%%%%%

\section{Proof of Theorem~\ref{T-INV-BOUS}}\label{Sec3}

	This Section is devoted to prove the exact controllability result in Theorem~\ref{T-INV-BOUS}.
	We will assume that Proposition~\ref{P-INV-BOUS} is satisfied and we will employ a scaling argument and the fact that, when $\kappa = 0$, \eqref{inv-bous} is time reversible.

	Let $T>0$, $\theta_0, \theta_1\in C^{2,\alpha}(\Ombar )$ and $\yvec_0, \yvec_1\in\Cvec(2,\alpha,\Gamma_0)$ be given.
	Let us see that, if
$$
	\|\yvec_0\|_{2,\alpha} + \|\yvec_1\|_{2,\alpha} + \|\theta_0\|_{2,\alpha} + \|\theta_1\|_{2,\alpha}
$$
	is small enough, we can construct a triplet $(\yvec,p,\theta)$ satisfying \eqref{inv-bous} and~\eqref{exact_condition}.
		
	If $\varepsilon\in (0,T/2)$ is small enough to have
$$
	\max \{\dis \varepsilon\| \yvec_0\|_{2,\alpha},\varepsilon^2\|\theta_0\|_{2,\alpha}\}\leq \delta \ \ 
	\text{(resp.}~\max \dis\{\varepsilon\| \yvec_1\|_{2,\alpha},~\varepsilon^2\|\theta_1\|_{2,\alpha}\}\leq \delta),
$$
	then, thanks to Proposition~\ref{P-INV-BOUS}, there exist $\dis (\yvec^0,\theta^0)$ in 
	$C^0([0,1]; \Cvec^{1,\alpha}(\overline\Om;\mathbb{R}^{N+1}))$ and a pressure $p^0$ (resp.~$(\yvec^1,\theta^1)$ and $p^1$) 
	solving \eqref{inv-bous}, with $\yvec^0(\xvec,0)\equiv\varepsilon \yvec_0(\xvec)$ and $\theta^0(\xvec,0)\equiv\varepsilon^2\theta_0(\xvec)$ 
	(resp.~$\yvec^1(\xvec,0)\equiv-\varepsilon \yvec_1(\xvec)$ and $\theta^1(\xvec,0)=\varepsilon^2\theta_1(\xvec)$) and satisfying  \eqref{cont-local}.

	Let us choose $\varepsilon$ of this form and let us introduce
	$\yvec:\Ombar \times [0,T]\mapsto \mathbb{R}^N$,	$p:\Ombar \times [0,T]\mapsto \mathbb{R}$ and
	$\theta:\Ombar \times [0,T]\mapsto \mathbb{R}$ as follows:
\[
	\begin{array}{ll}
		\noalign{\smallskip} \dis
		\left\{
			\begin{array}{l}
    				\noalign{\smallskip} \dis \yvec(\xvec,t) = \varepsilon^{-1}\yvec^0(\xvec,\varepsilon^{-1}t),           	\\
    				\noalign{\smallskip} \dis p(\xvec,t) = \varepsilon^{-2}p^0(\xvec,\varepsilon^{-1}t),               		\\
				\noalign{\smallskip} \dis \theta(\xvec,t) = \varepsilon^{-2}\theta^0(\xvec,\varepsilon^{-1}t),
			\end{array}
		\right.
		& \text{for }\ (\xvec,t) \in \Ombar \times[0,\varepsilon],                                                   	 			\\
		\\
		\noalign{\smallskip} \dis
		\left\{
			\begin{array}{l}
				\noalign{\smallskip} \dis     \yvec(\xvec,t) = \0vec,									\\
				\noalign{\smallskip} \dis     p(\xvec,t) = 0,								       			\\
				\noalign{\smallskip} \dis     \theta(\xvec,t) = 0,
			\end{array}
		\right.
		& \text{for }\ (\xvec,t) \in \Ombar \times(\varepsilon,T-\varepsilon),                                         			\\
	\end{array}
\]

\[
	\begin{array}{ll}
		\noalign{\smallskip} \dis
		\left\{
			\begin{array}{l}
				\noalign{\smallskip} \dis     \yvec(\xvec,t) = -\varepsilon^{-1}\yvec^1(\xvec,\varepsilon^{-1}(T-t)),  \\
				\noalign{\smallskip} \dis     p(\xvec,t) = \varepsilon^{-2}p^1(\xvec,\varepsilon^{-1}(T-t)),           \\
				\noalign{\smallskip} \dis     \theta(\xvec,t) = \varepsilon^{-2}\theta^1(\xvec,\varepsilon^{-1}(T-t)),
			\end{array}
		\right.
		& \text{for }\ (\xvec,t) \in \Ombar \times[T-\varepsilon,T].											
	\end{array}
\]
	Then, $(\yvec,\theta)\in C^0([0,T]; \Cvec^{1,\alpha}(\Ombar ;\mathbb{R}^{N+1})$ and the triplet $(\yvec,p,\theta)$ satisfies 
	\eqref{inv-bous} and~\eqref{exact_condition}.

%%%%%%%%%%%%%%%%%%%%%%%%%%%%%%%%%%%%%%%%%%
%%%% SECTION 4
%%%%%%%%%%%%%%%%%%%%%%%%%%%%%%%%%%%%%%%%%%

\section{Proof of Proposition \ref{P-INV-BOUS}. The 2D case}\label{Sec4}

	Let $\mu\in C^{\infty}([0,1])$ be a function such that $\mu\equiv1$ in $[0,1/4]$, $\mu\equiv0$ in $[1/2,1]$ and $0<\mu<1$.
	Proposition~\ref{P-INV-BOUS} is a consequence of the following result:
\begin{propo}\label{P1}
	There exists $\delta>0$ such that, if~$\max\left\{\|\yvec_0\|_{2,\alpha},\|\theta_0\|_{2,\alpha}\right\} \leq \delta$, then the coupled system
\begin{equation}\label{p-ext1}
	\left\{
		\begin{array}{lll}
		    	\noalign{\smallskip} \dis \zeta_t +\yvec\cdot \nabla\zeta = - \mathbf{k}\times \nabla \theta	&\text{in}&  \Om\times(0,1), 		\\
		    	\noalign{\smallskip} \dis \theta_t + \yvec\cdot \nabla\theta=  0							&\text{in}&  \Om\times(0,1), 		\\
		    	\noalign{\smallskip} \dis \nabla \cdot \yvec = 0, ~ \nabla\times \yvec  = \zeta 				&\text{in}&  \Om\times(0,1), 		\\
		    	\noalign{\smallskip} \dis \yvec\cdot\nvec= (\overline{\yvec}+\mu\, \yvec_0)\cdot \nvec			&\text{on}&  \Gamma\times (0,1),     	\\
		    	\noalign{\smallskip} \dis \zeta(0) = \nabla\times\yvec_0,~ \theta(0)=\theta_0				&\text{in}&  \Om,            			
		\end{array}
	\right.
\end{equation}
	possesses at least one solution $(\zeta,\theta,\yvec)$, with
\begin{equation}\label{reg}
	(\zeta,\theta,\yvec)\in C^0([0,1];C^{0,\alpha}(\overline\Om))\times C^0([0,1];C^{1,\alpha}(\overline\Om))
	\times C^0([0,1];\Cvec^{1,\alpha}(\overline\Om;\mathbb{R}^2)),
\end{equation} 	
	such that
\begin{equation}\label{27p}
	\theta(\xvec,t)=0 \quad \text{in} \quad \Omega\times(1/2,1) \quad \text{and} \quad \zeta(\xvec,1)=0 \quad \text{in} \quad \Omega.
\end{equation}
\end{propo}

	The reminder of this section is devoted to prove Proposition \ref{P1}.
	We are going to adapt some ideas from Bardos and Frisch \cite{Bardos} and Kato \cite{Kato}, already used in \cite{Coron3} and \cite{Glass1}.
	Let us give a sketch.	

	We will start from an arbitrary field $\zvec = \zvec(\xvec,t)$ in a suitable class $\Svec$ of continuous functions.
	To this $\zvec$, we will associate a scalar function $\theta$ 
	(a temperature) verifying
\[
	\left\{
\begin{array}{lll}
    \noalign{\smallskip}\dis \theta_t + \zvec\cdot\nabla\theta = 0 	&\hbox{in}&\Om\times(0,1),\\
    \noalign{\smallskip} \theta(\xvec,0) = \theta_0(\xvec)                	&\hbox{in}&\Om.
    \end{array}
    \right.
\]
	and
$$
	\theta(\xvec,t)=0 \quad \hbox{in}\quad \Om\times(1/2,1).
$$
	With the help of $\theta$, we will then construct a function $\zeta$
	(an associated vorticity) satisfying
\[
	\left\{
		\begin{array}{lll}
			\noalign{\smallskip} \dis \zeta_t +\zvec\cdot \nabla\zeta = -  \mathbf{k}\times \nabla \theta 		 &\hbox{in}& \Om\times(0,1), \\
			\noalign{\smallskip} \dis \zeta(0) = \nabla\times\yvec_0                                     					 &\hbox{in}& \Om.
		\end{array}
	\right.
\]
	and
$$
	\zeta(\xvec,1) = 0 \quad\hbox{in}\quad \Om.
$$
	Then, we will construct a field $\yvec=\yvec(\xvec,t)$ such that $\nabla \times \yvec = \zeta$ and $\nabla \cdot \yvec = 0$.
	 This way, we will have defined a mapping $F$ with $F(\zvec)=\yvec$.
	We will choose $\Svec$ such that $F$ maps $\Svec$ into itself and an appropriate extension of $F$ 
	possesses exactly one fixed-point~$\yvec$.
	Finally, it will be seen that the triplet $(\zeta,\theta,\yvec)$, where $\zeta$ and $\theta$ are respectively the vorticity and 
	temperature associated to $\yvec$, solves \eqref{p-ext1} and satisfies \eqref{reg}.

	Let us now give the details.

	The good definition of $\Svec$ is as follows.
	First, let us denote by $\Svec'$ the set of fields $\zvec\in C^0([0,1];\Cvec^{2,\alpha}(\overline\Om;\mathbb{R}^2))$ such that $\nabla\cdot\zvec=0$ in 	$\Om \times (0,1)$ and $\zvec\cdot\nvec=(\overline\yvec+\mu\yvec_0)\cdot\nvec$ on $\Gamma\times(0,1)$.
	Then, for any $\nu>0$, we set
$$
	\Svec_\nu := \{\, \zvec\in \Svec': \|\zvec-\overline\yvec\|_{0,2,\alpha}\leq\nu \,\}.
$$

	Let $\nu>0$ be the constant furnished by Lemma~\ref{nu-small} and let us carry out the previous process with $\Svec := \Svec_\nu$.
	To guarantee that $\Svec_\nu$ is nonempty, it suffices to assume that the initial data $\yvec_0$ is
	sufficiently small in $\Cvec^{2,\alpha}(\Ombar;\mathbb{R}^2)$, since, if this is the case, $\overline\yvec+\mu\yvec_0 \in \Svec_\nu$.

	Let $\zvec \in \Svec_\nu$ be given and let us set $\zvec^*=\yvec^* + \pi_2(\zvec-\overline{\yvec})$.
	We have the estimates
\begin{equation}\label{ybarra*}
	\|\zvec^*(\cdot\,,t)\|_{2,\alpha}\leq \|\yvec^*(\cdot\,,t)\|_{2,\alpha}+ C\|(\zvec-\overline\yvec)(\cdot\,,t)\|_{2,\alpha} \quad \forall t\in[0,1]
\end{equation}
	and the following result holds:
	
\begin{lemma}\label{L-X}
	The flux $\Zvec^*$ associated to $\zvec^*$ satisfies $\Zvec^*\in C^1([0,1]\times[0,1];\Cvec^{2,\alpha}(\Ombar_3;\mathbb{R}^2))$.
\end{lemma}

	Recall that $\Zvec^* = \Zvec^*(\mathbb{x},t,s)$ is, by definition, the unique function satisfying
\begin{equation}\label{X}
	\left\{
\begin{array}{ll}
\noalign{\smallskip} \dis \Zvec^*_t(\xvec,t,s)= \zvec^*(\Zvec^*(\xvec,t,s),t),\\
\noalign{\smallskip} \dis  \Zvec^*(\xvec,s,s)=\xvec,
\end{array}
	\right.
\end{equation}
	and
$$
	\Zvec^*(\xvec,t,s)\in \Ombar_3 \quad \forall (\xvec,t,s)\in \Ombar_3\times[0,1]\times[0,1].
$$

	For the proof of Lemma~\ref{L-X}, it suffices to apply directly the classical existence, uniqueness and regularity theory of ODEs.
	
	Since $\Zvec^*\in C^1([0,1]\times[0,1];\Cvec^{2,\alpha}(\Ombar_3;\mathbb{R}^2))$, 
	$\theta_0\in C^{2,\alpha}(\Ombar)$ and~$\pi_1$ maps continuously 
	$C^{2,\alpha}(\Ombar)$ into~$C^{2,\alpha}(\Ombar_3)$, there exists a unique solution 
	$\theta^*\in C^0([0,1/2];C^{2,\alpha}(\Ombar_3))$ to the problem
\begin{equation}\label{p-extL55}
	\left\{
		\begin{array}{lll}
			\noalign{\smallskip}\dis \theta^*_t + \zvec^*\cdot\nabla\theta^* = 0    	 &\hbox{in}&\Om_3 \times (0,1/2),\\
			\noalign{\smallskip} \theta^*(\xvec,0)=\pi_1(\theta_0)(\xvec)		&\hbox{in}& \Om_3.
		\end{array}
	\right.
\end{equation}	

	Note that,  in \eqref{p-extL55}, no boundary condition on $\theta^*$ is needed.
	Obviously, this is because~$\Supp \zvec^* \subset \Om_3$.

	The solution to \eqref{p-extL55} verifies $\Supp \theta^*(\cdot\,,t) \subset \Zvec^*(\Om_2,t,0)$ for all  $t \in [0,1/2]$.
	In particular, in view of the choice of $\nu$, we get:
\[
    \Supp \theta^*(\cdot\,,1/2) \subset \Zvec^*(\Om_2,1/2,0)\subset\Om_3\setminus\Ombar _2,
\]
	whence $\theta^*(\xvec,1/2) = 0$ in $\Om_2$. 

	Let $\theta$ be the following function:
\[
	\theta(\xvec,t)=
\left\{
	\begin{array}{ll}
		\theta^*(\xvec,t), 	&(\xvec,t)\in \Ombar \times [0,1/2),\\
		0,                 		&(\xvec,t)\in \Ombar \times [1/2,1].
	\end{array}
\right.
\]
	Then $\theta\in C^0([0,1];C^{2,\alpha}(\Ombar))$ and one has
\begin{equation}\label{p-theta}
	\left\{
\begin{array}{lll}
	\noalign{\smallskip}\dis \theta_t + \zvec\cdot\nabla\theta = 0     &\hbox{in}&\Om\times(0,1),\\
	\noalign{\smallskip} \theta(\xvec,0)=\theta_0(\xvec)                   &\hbox{in}&\Om.
    \end{array}
	\right.
\end{equation}

	For the construction of $\zeta$, the argument is the following.
	First, let us introduce $\zeta_0^*:=\nabla\times(\pi_2(\yvec_0))$ and let $\zeta^*\in C^0([0,1/2];C^{1,\alpha}(\Ombar_3))$ be 
	the unique solution to the problem   
\[
\left\{
\begin{array}{lll}
	\noalign{\smallskip} \dis
	\zeta^*_t +\zvec^*\cdot \nabla\zeta^* = - \mathbf{k} \times \nabla\theta^*		  		&\hbox{in}& \Om_3 \times (0,1/2), \\
	\noalign{\smallskip} \dis \zeta^*(\xvec,0) = \zeta_0^*(\xvec)                                    			 &\hbox{in}& \Om_3.
\end{array}
\right.
\]
	With this $\zeta^*$, we define $\zeta_{1/2}\in C^{1,\alpha}(\Ombar)$ with
	\[
\zeta_{1/2}(\xvec):=\zeta^*(\xvec,1/2)~\,\hbox{forall}~\,\xvec \in \Ombar.
	\]
	Then, let $\zeta^{**}\in C^0([1/2,1]; C^{1,\alpha}(\Ombar_3))$ be the unique solution to the problem
\[
\left\{
\begin{array}{lll}
	\noalign{\smallskip} \dis \zeta^{**}_t +\zvec^*\cdot \nabla\zeta^{**} = 0               	&\hbox{in}& \Om_3 \times (1/2,1), \\
	\noalign{\smallskip} \dis \zeta^{**}(\xvec,1/2) = \pi_1(\zeta_{1/2})(\xvec)      	&\hbox{in}& \Om_3.
\end{array}
\right.
\]
	We have $\dis \zeta^{**}(\Zvec^*(\xvec,t,1/2),t) = \pi_1(\zeta_{1/2})(\xvec)$ for all $(\xvec,t) \in  \Ombar _3\times[1/2,1]$ 
	and, again from the choice of~$\nu$,
\[
    \Supp \zeta^{**}(\cdot\,,1) \subset \Zvec^*(\Om_2,1,1/2)\subset\Om_3 \setminus \Ombar _2
\]
	and $\zeta^{**}(\mathbf{x},1)\equiv0$ in $\Om_2$.
	Finally, we can define $\zeta\in C^0([0,1]; C^{1,\alpha}(\Ombar))$, with
\[
	\zeta(\xvec,t)=
\left\{
\begin{array}{ll}
	\zeta^*   (\xvec,t),        	&(\xvec,t)\in \Om \times (0,1/2),\\
	\zeta^{**}(\xvec,t),        	&(\xvec,t)\in \Om \times [1/2,1).
\end{array}
\right.
\]

	 Obviously, $\zeta$ is a solution to the initial-value problem
\begin{equation}\label{p-zeta}
\left\{
\begin{array}{lll}
	\noalign{\smallskip}\dis\zeta_t + \zvec\cdot\nabla\zeta = -\mathbf{k}\times\nabla \theta &\hbox{in}& \Om \times (0,1), \\
	\noalign{\smallskip} \dis \zeta(\xvec,0) = (\nabla\times\yvec_0)(\xvec)                                     &\hbox{in}& \Om.
\end{array}
\right.
\end{equation}
	
	With this $\zeta$, we can now get a unique $\yvec\in C^0([0,1];\Cvec^{2,\alpha}(\Ombar;\mathbb{R}^2))$ such that 
	$\nabla\times\yvec = \zeta$ in~$\Om \times (0,1)$, $\nabla\cdot\yvec=0$ in $\Om \times (0,1)$ and~$\yvec\cdot\nvec = (\overline\yvec+\mu\yvec_0)\cdot\nvec$ on $\Gamma\times[0,1]$.
	Indeed, let $\psi\in C^0([0,1]; C^{3,\alpha}(\Ombar))$ be the unique solution to the following family of elliptic equations:
\begin{equation}\label{Phi}
	\left\{
\begin{array}{lll}
	\noalign{\smallskip}
	\dis -\Delta \psi = \zeta-\mu\nabla\times\yvec_0 	&~\hbox{in}~\Om\times(0,1),   \\
	\noalign{\smallskip}
	\dis         \psi = 0                              				&~\hbox{on}~\Gamma\times(0,1).
\end{array}
	\right.
\end{equation}
	Then, let us set $\yvec:=\nabla\times\psi+\overline\yvec+\mu\yvec_0$.
	We have that $\yvec \in C^0([0,1]; \Cvec^{2,\alpha}(\Ombar;\mathbb{R}^2))$ and satisfies the required properties.
		Since $\yvec$ is determined by $\zvec$, we write $\yvec = F(\zvec)$.
	Accordingly, $F: \Svec_\nu \mapsto \Svec'$ is well defined.

	The following result holds:
	
\begin{lemma}\label{Sr}
	There exists $\delta>0$ such that, if
	\begin{equation}\label{47p}
\max\left\{\|\yvec_0\|_{2,\alpha},\|\theta_0\|_{2,\alpha}\right\} \leq \delta,
	\end{equation}
then $ F(\Svec_\nu)\subset \Svec_\nu$.
\end{lemma}

\begin{proof}
	Let $\zvec\in \Svec_\nu$ be given.
	Then $F(\zvec)-\overline \yvec=\nabla\times \psi+\mu\yvec_0$ and we have:
\[
	\|F(\zvec)(\cdot\,,t)-\overline \yvec(\cdot\,,t)\|_{2,\alpha}\leq  C(\|\zeta(\cdot\,,t)\|_{1,\alpha} + \|\yvec_0\|_{2,\alpha}).
\]
	Applying Lemma \ref{lemma Bardos} to the equations of $\theta^*$ and $\zeta^*$, we get
\begin{equation}\label{theta*}
	\|\theta^*(\cdot\,,t)\|_{2,\alpha}\leq \|\pi_1(\theta_0)\|_{2,\alpha}\exp\left({\dis K\int_0^t\|\zvec^*(\cdot\,,\tau)\|_{2,\alpha}\,d\tau}\right)
\end{equation}
and
\begin{equation}\label{estzeta}
	\|\zeta^*(\cdot\,,t)\|_{1,\alpha}\leq C(\|\pi_2(\yvec_0)\|_{2,\alpha}+\|\pi_1(\theta_0)\|_{2,\alpha})
	\exp\bigg(\dis K\int_0^t\|\zvec^*(\cdot\,,\tau)\|_{2,\alpha}\,d\tau\bigg).
\end{equation}
	With similar arguments, we also obtain 
\begin{equation}\label{zeta**}
	\|\zeta^{**}(\cdot\,,t)\|_{1,\alpha}\leq  C(\|\pi_2(\yvec_0)\|_{2,\alpha}+\|\pi_1(\theta_0)\|_{2,\alpha})
	\exp\bigg(\dis K\int_0^t\|\zvec^*(\cdot\,,\tau)\|_{2,\alpha}(\tau)d\tau\bigg)
\end{equation}
	for all $t\in[1/2,1]$.
	Thanks to \eqref{estzeta} and \eqref{zeta**}, we get the following for $\zeta$:
\begin{equation}\label{zeta}
	\|\zeta(\cdot\,,t)\|_{1,\alpha}\leq C(\|\yvec_0\|_{2,\alpha}+\|\theta_0\|_{2,\alpha})\exp\bigg(\dis K\int_0^t\|\zvec^*(\cdot,\tau)\|_{2,\alpha}d\tau\bigg).
\end{equation}

 	Using \eqref{zeta}, \eqref{ybarra*} and the definition of $\Svec_\nu$, we see that
\[
\begin{alignedat}{2}
	\|F(\zvec)(\cdot\,,t)-\overline \yvec(\cdot\,,t)\|_{2,\alpha}\leq&~ C_1(\|\yvec_0\|_{2,\alpha}+\|\theta_0\|_{2,\alpha})
	\exp\left(C_2\int_0^t\|\zvec(\cdot\,,\tau)-\overline\yvec(\cdot\,,\tau)\|_{2,\alpha}\,d\tau\right)\\
	\leq&~C_1(\|\yvec_0\|_{2,\alpha}+\|\theta_0\|_{2,\alpha})\exp(C_2\nu).
\end{alignedat}
\]

	Let $\delta>0$ be such that $2 C_1 \delta e^{C_2 \nu} \leq \nu$ and let us assume that~\eqref{47p} is satisfied.
	Then
\[
	\|F(\zvec)-\overline\yvec\|_{0,2,\alpha}\leq \nu
\]
	and, consequently, $F$ maps $\Svec_\nu$ into itself.
\end{proof}

\

	Now, we will prove the existence and uniqueness of a fixed-point of the extension of $F$ in the closure of $\Svec_\nu$ in 
	$C^0([0,1];\Cvec^{1,\alpha}(\Ombar;\mathbb{R}^3))$. 
	For this purpose, we will check that $F$ satisfies the hypotheses of Theorem \ref{fixedpoint}.

	To this end, we will first establish two important lemmas.
	The first one is the following:
	
\begin{lemma}\label{L-X-est}
	There exists $\widetilde C>0$, only depending on $\|\yvec_0\|_{2,\alpha}$, $\|\theta_0\|_{2,\alpha}$ and $\nu$, such that, for
	any $\zvec^1,~\zvec^2\in \Svec_\nu$, one has:
\begin{equation}\label{eq-fix-p}
	\|(\zeta^1-\zeta^2)(\cdot,t)\|_{0,\alpha}\leq\widetilde C \int_0^t\|(\zvec^1-\zvec^2)(\cdot,s)\|_{1,\alpha} \,ds
	\quad \forall t \in [0,1],
\end{equation}
	where $\zeta^i$ is the vorticity associated to $\zvec^i$.
\end{lemma}
\begin{proof}
	First of all, let us introduce $\wvec^*:=\zvec^{*,1}-\zvec^{*,2}$ and $\Theta^* := \theta^{*,1}-\theta^{*,2}$
	(where the notation is self-explaining). Obviously, the estimates \eqref{ybarra*} and (resp. \eqref{theta*} and 
	\eqref{estzeta}) hold for $\zvec^{*,1}$ and $\zvec^{*,2}$ (resp. $\theta^{*,1}$ and $\theta^{*,2}$ and $\zeta^{*,1}$ 
	and $\zeta^{*,2}$). Furthermore, it is clear that
\[
	\Theta^*_t+\zvec^{*,1}\cdot\nabla\Theta^*=-\wvec^*\cdot\nabla\theta^{*,2}.
\]
	
	Applying Lemma~\ref{lemma Bardos1} to this equation, we have
\begin{equation}\label{Theta^*}
\begin{alignedat}{2}
	\frac{d}{dt^+}\|\Theta^*(\cdot,t)\|_{1,\alpha}\leq
	& ~\|\wvec^*(\cdot,t)\|_{1,\alpha}\|\theta^{*,2}(\cdot,t)\|_{2,\alpha}
	+K\|\zvec^{*,1}(\cdot,t)\|_{1,\alpha}\|\Theta^*(\cdot,t)\|_{1,\alpha}.
\end{alignedat}
\end{equation}

		In view of Gronwall's Lemma,  \eqref{ybarra*} and \eqref{theta*}, we see that
\begin{equation}\label{Theta^*est}
\begin{alignedat}{2}
	\|\Theta^*(\cdot,t)\|_{1,\alpha}\leq& ~\widetilde C_0\int_0^t\|\wvec^*(\cdot,s)\|_{1,\alpha}\,ds\quad \forall t \in [0,1/2].
\end{alignedat}
\end{equation}

	The equations verified by $\Upsilon^*:=\zeta^{*,1}-\zeta^{*,2}$ and $\Upsilon^{**}:=\zeta^{**,1}-\zeta^{**,2}$ are
\[
	\Upsilon^*_t+\zvec^{*,1}\cdot\nabla\Upsilon^*=-\wvec^*\cdot\nabla\zeta^{*,2}-\mathbf{k} \times\nabla\Theta^*
\]
	and
\[
	\Upsilon^{**}_t+\zvec^{*,1}\cdot\nabla\Upsilon^{**}=-\wvec^{*}\cdot\nabla\zeta^{**,2},
\]
	respectively.
	Consequently, applying Lemma~\ref{lemma Bardos1} to these equations, we get:
\begin{equation}\label{est-Psi*}
	\begin{alignedat}{2}
		\frac{d}{dt^+}\|\Upsilon^*(\cdot,t)\|_{0,\alpha}\leq
	     	&~~\|(\wvec^*\cdot\nabla\zeta^{*,2}+\mathbf{k} \times\nabla\Theta^*)(\cdot,t)\|_{0,\alpha}
	     	+K \|\zvec^{*,1}(\cdot,t)\|_{1,\alpha}\|\Upsilon^*(\cdot,t)\|_{0,\alpha}
	\end{alignedat}
\end{equation}
	and
\begin{equation}\label{est-Psi**}
	\begin{alignedat}{2}
		\frac{d}{dt^+}\|\Upsilon^{**}(\cdot,t)\|_{0,\alpha}\leq
		& ~ \|(\wvec^*\cdot\nabla\zeta^{**,2})(\cdot,t)\|_{0,\alpha}
	         +K \|\zvec^{*,1}(\cdot,t)\|_{1,\alpha}\|\Upsilon^{**}(\cdot,t)\|_{0,\alpha}.
	\end{alignedat}
\end{equation}

	Applying Gronwall's Lemma, we deduce in view of \eqref{Theta^*est} that
\[
\begin{alignedat}{2}
	\|\Upsilon^*(\cdot,t)\|_{0,\alpha}\leq
	& ~\widetilde C_1\|\zeta^{*,2}\|_{0,1,\alpha}\int_0^t\|\wvec^*(\cdot,s)\|_{1,\alpha}\,ds \quad \forall t \in [0,1/2]	
\end{alignedat}
\]
	and
\[
	\|\Upsilon^{**}(\cdot,t)\|_{0,\alpha} \leq~\widetilde C_2\|\zeta^{*,2}\|_{0,1,\alpha}
	\int_0^t\|\wvec^*(\cdot,s)\|_{1,\alpha}\,ds \quad \forall t \in [1/2,1].
\]

	Finally, we see from these estimates and \eqref{zeta} that \eqref{eq-fix-p} holds.
\end{proof}

\	
	
	Note that $\yvec^1-\yvec^2=\nabla\times(\psi^1-\psi^2)$, whence $\nabla\times(\nabla\times(\psi^1-\psi^2))=\zeta^1-\zeta^2$ 
	and~$\nabla\times(\psi^1-\psi^2)\cdot\nvec=0$ on $\Gamma \times [0,1]$. 
	
	Let us denote by $\Mvec$ the set of fields 
	$\wvec\in C^0([0,1];\Cvec^{1,\alpha}(\overline\Om;\mathbb{R}^2))$ such that $\nabla\cdot\wvec=0$ in $\Om \times (0,1)$ 
	and $\wvec\cdot\nvec=0$ on $\Gamma\times(0,1)$. Note that, for any $\wvec \in \Mvec$, 
	the norms $\| \wvec \|_{1,\alpha}$ and~$\| \nabla \times\wvec \|_{0,\alpha}$ are equivalent;
	we will set in the sequel $|||\wvec|||_{1,\alpha} := \|\nabla \times\wvec\|_{0,\alpha}$ for any $\wvec \in \Mvec$.
	
\begin{lemma}\label{iteracion}
	Let $\tilde C $ be the constant furnished by Lemma~$\ref{L-X-est}$.
	For any $\zvec^1, \zvec^2\in \Svec_\nu$, one has
\begin{equation}\label{Fm}
	|||(F^m(\zvec^1)-F^m(\zvec^2))(\cdot,t)|||_{1,\alpha} \leq \frac{(\tilde C t)^m}{m!}
	\|\zvec^1-\zvec^2\|_{0,1,\alpha}  \quad \forall m \geq 1.
\end{equation}
\end{lemma}

\begin{proof}
	The proof is by induction.
	
	For $m=1$, this is obvious, in view of Lemma~\ref{L-X-est}.
	
	Let us assume that \eqref{Fm} holds for $m=k$.
	Applying Lemma~\ref{L-X-est} to $\yvec^1=F^{k}(\zvec^1)$ and $\yvec^2=F^{k}(\zvec^2)$, we have
\[
	|||(F(\yvec^1)-F(\yvec^2))(\cdot,t)|||_{1,\alpha} \leq \tilde C \int_0^t\|(\yvec^1-\yvec^2)(\cdot,s)\|_{1,\alpha}\,ds \quad \forall t \in [0,1].
\]
	Therefore, using the induction hypothesis, we obtain:
\[
	\begin{alignedat}{2} \dis
		|||(F^{k+1}(\zvec^1)-F^{k+1}(\zvec^2))(\cdot,t)|||_{1,\alpha}&\leq \tilde C 
		\|\zvec^1-\zvec^2\|_{0,1,\alpha} \int_0^t\frac{(\tilde C s)^k}{k!} \, ds				\\ 
		\noalign{\smallskip} \dis &= \frac{(\tilde C t)^{k+1}}{(k+1)!} 
		\|\zvec^1-\zvec^2\|_{0,1,\alpha}
	\end{alignedat}
\]
	
	This ends the proof.
\end{proof}

\

	We deduce that, for some $\widehat C > 0$, any $m \geq 1$ and any $\zvec^1, \zvec^2 \in \Svec_\nu$, one has
\[
	 \max_{t\in[0,1]}\|(F^m(\zvec^1)-F^m(\zvec^2))(\cdot,t)\|_{1,\alpha} \leq \frac{\widehat C\tilde C^m}{m!}
	\left( \max_{\tau\in[0,1]} \|(\zvec^1-\zvec^2)(\cdot,\tau)\|_{1,\alpha} \right).
\]
	Consequently, if $m$ is large enough, $F^m: \Svec_\nu \mapsto \Svec_\nu$ is a contraction, that is, there exists $\gamma \in (0,1)$ such that
\begin{equation}\label{FFm}
	\| F^m(\zvec^1) - F^m(\zvec^2) \|_{0,1,\alpha} \leq \gamma \|\zvec^1-\zvec^2\|_{0,1,\alpha} \ \ \forall \zvec^1, \zvec^2 \in \Svec_\nu.
\end{equation}

	Thus, we can apply Theorem~\ref{fixedpoint} with
\[
	B_1 = C^0([0,1];\Cvec^{1,\alpha}(\Ombar;\mathbb{R}^2)),\,\, B_2 = C^0([0,1];\Cvec^{2,\alpha}(\Ombar;\mathbb{R}^2)),
	\,\, B = \Svec_\nu \,\, \text{ and } \,\, G = F
	\]
	and deduce that $F$ possesses a unique extension $\widetilde F$ with a unique fixed-point $\yvec$ in the closure of 
	$\Svec_\nu$ in~$C^0([0,1];\Cvec^{1,\alpha}(\Ombar;\mathbb{R}^2))$.
	
	It is easy to check that $\yvec$ is, together with some $\zeta$ and~$\theta$, a solution to~\eqref{p-ext1} 
	satisfying~\eqref{reg} and~\eqref{27p}.

	This ends the proof.

%%%%%%%%%%%%%%%%%%%%%%%%%%%%%%%%%%%%%%%%%%
%%%% SECTION 5
%%%%%%%%%%%%%%%%%%%%%%%%%%%%%%%%%%%%%%%%%%

\section{Proof of Proposition~\ref{P-INV-BOUS}. The 3D case}\label{Sec5}

	In this Section we are going to prove Proposition~\ref{P-INV-BOUS} in the three-dimensional case.
	
	The situation is nor exactly the same considered in the previous section, since the vorticity
	(which was fundamental for the fixed-point argument) is now a field and not a scalar variable.
	
	Let he times $t_i$, the balls $B^i$ and the functions $\rho^i$ be as in~Section~\ref{sec-3-traj} and let us set  $\omvec_0=\nabla\times \pi_3(\yvec_0)$.
	Proposition~\ref{P-INV-BOUS} is a consequence of the following result:	
\begin{propo}\label{P2}
	There exists $\delta>0$ such that, if $\max\left\{\|\yvec_0\|_{2,\alpha},\|\theta_0\|_{2,\alpha}\right\} \leq  \delta$,
	then the coupled system
\begin{equation}\label{p-ext1dim3}
	\left\{
		\begin{array}{lll}
    			\noalign{\smallskip} \dis \omvec_t +( \yvec\cdot \nabla)\omvec= (\omvec\cdot\nabla)\yvec
    								- \mathbf{k} \times \nabla \theta     						&\text{in}&  \Om \times (0,1), \\
    			\noalign{\smallskip} \dis \theta_t + \yvec\cdot \nabla\theta=  0             					&\text{in}&  \Om \times (0,1), \\
   			\noalign{\smallskip} \dis \nabla \cdot \yvec = 0,~ \nabla\times \yvec  = \omvec 				&\text{in}&  \Om \times (0,1), \\
    			\noalign{\smallskip} \dis \yvec\cdot \nvec= (\overline{\yvec}+\mu\yvec_0)\cdot \nvec			&\text{on}&  \Gamma \times (0,1),   \\
    			\noalign{\smallskip} \dis \omvec(0) = \nabla\times\yvec_0,~\theta(0)=\theta_0    				&\text{in}&  \Om
		\end{array}
	\right.
\end{equation}
	possesses at least one solution $(\omvec,\theta,\yvec)$, with
\begin{equation}\label{regom}
	(\omvec,\theta,\yvec)\in  C^0([0,1];\Cvec^{0,\alpha}(\overline\Om;\mathbb{R}^3))\times C^0([0,1];C^{1,\alpha}(\overline\Om))
	\times C^0([0,1];\Cvec^{1,\alpha}(\overline\Om;\mathbb{R}^3)),
\end{equation} 	
	such that
\begin{equation}\label{27pom}
	\theta(\xvec,t)=0 \quad \text{in} \quad \Omega\times(t_{k-1/2},1) \quad \text{and} 
	\quad \omvec(\xvec,t)=0 \quad \text{in} \quad \Omega \times (t_{2k-1/2},1).
\end{equation}
\end{propo}

	Let us give the proof of this result.
	We will repeat the strategy of proof of Proposition~\ref{P1}, but we will have to incorporate some ideas from Bardos and~Frisch~\cite{Bardos} and~Glass~\cite{Glass3};
	this is mainly due to the complexity of the field $\mathbf{y}^*$ in this case.
	We will use the notation in~Section~\ref{sec-3-traj}.
		
	First, let us denote by $\Rvec'$ the set of fields $\zvec\in C^0([0,1];\Cvec^{2,\alpha}(\overline\Om;\mathbb{R}^3))$ such that $\nabla\cdot\zvec=0$ in 	$\Om \times (0,1)$ and $\zvec\cdot\nvec=(\overline\yvec+\mu\yvec_0)\cdot\nvec$ on $\Gamma\times(0,1)$.
	Then, for any $\nu>0$, we set
$$
	\Rvec_\nu = \{\, \zvec\in \Rvec': \|\zvec-\overline\yvec\|_{0,1,\alpha}\leq\nu\,\}.
$$

	Let $\nu>0$ be the constant furnished by Lemma~\ref{Lemmalocal}.
	As before, if the initial datum $\yvec_0$ is sufficiently small in $\Cvec^2(\Ombar;\mathbb{R}^3)$, the set $\Rvec_\nu$ is nonempty.

	Now, we are going to construct a mapping $F: \Rvec_\nu\rightarrow \Rvec_\nu$.
	
	We start from an arbitrary $\zvec \in \Rvec_\nu$ and we set $\zvec^*:= \yvec^* + \pi_3(\zvec - \overline{\yvec})$.
	Then, we denote by $\theta^*$ the unique solution to
\[
	\left\{
		\begin{array}{lll}
			\theta^*_t+\zvec^*\cdot\nabla\theta^*=0                                  			& \hbox{in }& \overline{\mathscr{O}} \times [0,1/2], \\
			\theta^*(\xvec,0)=\sum_{i=1}^{k}\psi^i(\xvec) \, \pi_1(\theta_0)(\xvec)          & \hbox{in }&  \overline{\mathscr{O}}.
		\end{array}
	\right.
\]
	Obviously, $\theta^*=\sum_{i=1}^{k}\theta^{i}$, where $\theta^i$ is the unique solution to
\begin{equation}\label{thetai}
	\left\{
	\begin{array}{ll}
	\theta^i_{t}+\zvec^*\cdot\nabla\theta^i=0                & \hbox{in } \ \overline{\mathscr{O}} \times [0,1/2], \\
	\theta^i(\xvec,0)=\psi^i(\xvec) \, \pi_1(\theta_0)(\xvec)          & \hbox{in } \ \overline{\mathscr{O}} .
	\end{array}
	\right.
\end{equation}

	The identities 
$$
	\theta^{i}(\Zvec^*(\xvec,t,0),t) = \psi^{i}(\xvec)\pi_1(\theta_0)(\xvec)\quad \forall(\xvec,t)\in \overline{\mathscr{O}} \times [0,1/2]
$$ 
imply that
\[
	\text{Supp\,}~\theta^{i}(\cdot\,,t)\subset\Zvec^*(B^i,t,0)\quad \forall t\in[0,1/2]. 
\]
	Hence, in view of Lemma~\ref{Lemmalocal}, we deduce that
\[
	\text{Supp\,}~\theta^{i}(\cdot\,,t_{i-1/2})\subset\Zvec^*(B^i,t_{i-1/2},0)\subset \mathscr{O}\setminus\overline{\mathscr{O}}_0,
\]
	whence 
\begin{equation}\label{thetai=0}
	\theta^i(\cdot\,,t_{i-1/2})=0\quad \hbox{in}\quad \Ombar.
\end{equation}

	Now, we set $\hat\theta(\xvec,t):=\theta^*(\xvec,t)$ in $\overline{\mathscr{O}} \times[0,t_{0}]$ and we say that, in~$\overline{\mathscr{O}}\times [t_0,1/2]$, $\hat\theta$~is the unique solution to
\begin{equation}\label{eqthetatheta}
	\left\{
		\begin{array}{lll} \dis
		\hat\theta_t+\zvec^*\cdot\nabla\hat\theta=0 \!&\hbox{in}&\!
		\overline{\mathscr{O}}\times\left( [t_0,1/2]\setminus\bigcup\limits_{i=1}^{k}\{ t_{i-{1\over2}}\}\right),\\ 
		\dis
		\hat\theta(\xvec,t_{i-1/2}) = \sum\limits_{l=i}^{k}\theta^{l}(\xvec,t_{i-1/2})-\theta^{i}(\xvec,,t_{i-1/2}) 
		\!&\hbox{in}&\!\overline{\mathscr{O}} ,\, 1 \leq i \leq k.
	\end{array}
	\right.
\end{equation}

	We notice that $\hat\theta(\cdot\,,t_{k-1/2})\equiv0$ in $\overline{\mathscr{O}}$. 
	Hence, $\hat\theta\equiv0$ in $\overline{\mathscr{O}}\times [t_{k-1/2},1/2]$.
	Moreover,
\[
	\hat\theta(\xvec,t) =
	\sum\limits_{l=i}^{k}\theta^{l}(\xvec,t)-\theta^{i}(\xvec,t)~\hbox{ in }~\overline{\mathscr{O}}\times(t_{i-1/2},t_{i+1/2}), 
	\ \ 1 \leq i \leq k-1.
\]
	Note that the lateral limits of $\hat \theta$ at the points $\{t_{i-1/2}\}_{i=1}^k$ are not necessarily the same in the whole domain 
	$\overline{\mathscr{O}}$. 
	
	Let $\theta$ be the restriction of $\hat \theta$ to $\Ombar$.
	Due to \eqref{thetai=0} and \eqref{eqthetatheta}, we see that $\theta$ is continuous at the $t_{i-1/2}$ with~$i=1, \dots k$, satisfies
\begin{equation}\label{theta}
\left\{
\begin{array}{ll}
	\theta_{t}+\zvec\cdot\nabla\theta=0 		& \hbox{in} \ \Om \times (0,1/2), \\
	\theta(\xvec,0)=\theta_0(\xvec)                  & \hbox{in} \ \Om
\end{array}
\right.
\end{equation}
	and belongs to $C^0([0,1];C^{1,\alpha}(\overline\Om))$.

	In a similar way, we can introduce a field $\widehat\omvec$ in $\overline{\mathscr{O}} \times [0,1]$ whose restriction to $\Omega$ is a function $\omvec$ satisfying the first PDE in~\eqref{p-ext1dim3} with $\mathbf{y}$ replaced by~$\mathbf{z}$.
	The definition of $\widehat \omvec$ will be made in three parts, respectively associated to the three time intervals $[0,1/2)$, $[1/2,t_{k+1/2})$ and $[t_{k+1/2},1]$.

	Let us introduce $\omvec_0:=\nabla\times(\pi_3(\yvec_0))$ and let $\omvec^*$ be the solution to
\[
	\left\{
	\begin{array}{lll}
	\omvec^*_{t}+(\zvec^*\cdot\nabla)\omvec^* = (\omvec^*\cdot\nabla)\zvec^*-(\nabla\cdot \zvec^*)\omvec^*
	- \overrightarrow{\textbf{k}}\times\nabla\pi_1(\theta)  &\hbox{in}& \mathscr{O} \times (0,1/2), \\
	\noalign{\smallskip} \dis
	\omvec^*(\xvec,0) = \omvec_0(\xvec)                         &\hbox{in}& \mathscr{O}.                \\
\end{array}
\right.
\]
	
	With this $\omvec^*$, we set $\omvec^{**}_{1/2}\in \Cvec^{1,\alpha}(\Ombar)$ with $\omvec^{**}_{1/2}(\xvec):=\omvec^*(\xvec\,,1/2)$~for all~$\xvec\in \Ombar$.
	Let us consider the solution $\omvec^{**}$ to the problem
\begin{equation}\label{om-hat}
	\left\{
	\begin{array}{lll}
	\omvec^{**}_t+(\zvec^*\cdot\nabla)\omvec^{**} =
	(\omvec^{**}\cdot\nabla)\zvec^*-(\nabla\cdot\zvec^*)\omvec^{**} &\hbox{in}& \mathscr{O} \times (1/2,1),\\
	\omvec^{**}(\xvec,1/2) = \sum\limits_{i=1}^k\psi^i(\xvec) \,\pi_3(\omvec^{**}_{1/2})(\xvec) &\hbox{in}& \mathscr{O}.
  \end{array}
    \right.
\end{equation}
	As before, we can decompose $\omvec^{**}$ as a sum of functions.
	More precisely, let $\omvec^1, \dots,\omvec^k$ be the solutions to the problems
\begin{equation}\label{eqomi}
\left\{
\begin{array}{lll}
	\omvec^i_{t}+(\zvec^*\cdot\nabla)\omvec^i =
	(\omvec^i\cdot\nabla)\zvec^*-(\nabla\cdot\zvec^*)\omvec^i &\hbox{in}& \mathscr{O} \times (1/2,1),\\
	\omvec^i(\xvec,1/2) = \psi^i(\xvec) \,\pi_3(\omvec^{**}_{1/2})(\xvec) &\hbox{in}& \mathscr{O} .
  \end{array}
    \right.
\end{equation}
	Then
$$
	\omvec^{**}=\sum\limits_{i=1}^{k}\omvec^i\quad\hbox{ in }\quad\overline{\mathscr{O}} \times [1/2,1].
$$
	Each $\omvec^i$ satisfies
\[
    \omvec^i(\Zvec^*(\xvec,t,1/2),t) \!=\! \omvec^i(\xvec,1/2)
    \!+\! \int_{1/2}^t[(\omvec^i\cdot\nabla)\zvec^* \!-\! (\nabla\cdot\zvec^*)\omvec^i](\Zvec^*(\xvec,\sigma,1/2),\sigma)\,d\sigma.
\]
	Consequently,
\[
    |\omvec^i(\Zvec^*(\xvec,t,1/2),t)|\leq|\omvec^i(\xvec,1/2)|
    +C\|\zvec^*\|_{0,1,0}\int_{1/2}^t|\omvec^i(\Zvec^*(\xvec,\sigma,1/2),\sigma)|\,d\sigma.
\]

	Notice that, if $\xvec\not\in B^i$ we then have
\[
    |\omvec^i(\Zvec^*(\xvec,t,1/2),t)|\leq C\|\zvec^*\|_{0,1,0}\int_{1/2}^t|\omvec^i(\Zvec^*(\xvec,\sigma,1/2),\sigma)|\,d\sigma
\]
	and, from {\it Gronwall's Lemma,} we necessarily have
\[
\begin{array}{c}
	\omvec^i(\Zvec^*(\xvec,t,1/2),t)=0 \quad \forall (\xvec,t)\in (\overline{\mathscr{O}} \setminus B^i)\times[1/2,1].
\end{array}
\]

	A consequence is that $\Supp~\omvec^i(\cdot\,,t) \subset\Zvec^*(B^i,t,1/2)$, whence we 
	get 
$$
	\omvec^i(\xvec,t_{k+i-1/2})=0 \quad\hbox{for all}\quad\xvec\in\Ombar.
$$

	Then, we simply set $\widehat\omvec(\xvec,t):=\omvec^*(\xvec,t)$ in $\overline{\mathscr{O}} \times[0,1/2]$ and $\widehat\omvec(\xvec,t):=\omvec^{**}(\xvec,t)$ in $\overline{\mathscr{O}} \times[1/2,t_{k+1/2}]$ and we say that, in $\mathscr{O} \times [t_{k+1/2},1]$, $\widehat\omvec$ is the unique solution to
\begin{equation}\label{eqomom}
	\left\{
	\begin{array}{lll}	
	{\widehat\omvec}_{t} \!+\! (\zvec^*\cdot\nabla){\widehat\omvec} \!=\!
	({\widehat\omvec}\cdot\nabla)\zvec^* \!-\! (\nabla\cdot\zvec^*){\widehat\omvec}
	&\!\!\!\hbox{in }&\!\!\!\!\!\!\overline{\mathscr{O} }\!\times\!\left([t_{k \!+\! 1/2},1] \!\setminus\! \bigcup\limits_{i=1}^{k} \{t_{k \!+\! i \!-\! 1/2}\}\right)
	\\ \dis
	\widehat\omvec(\xvec,t_{k \!+\! i \!-\! \frac{1}{2}}) \!=\! \sum\limits_{l \! =\! i}^{k}\widehat\omvec^{l}(\xvec,t_{k \!+\! i \!-\! \frac{1}{2}}) \!-\! \widehat\omvec^{i}(\xvec,t_{k \!+\! i \!-\! \frac{1}{2}})
	&\!\!\!\hbox{in }&\!\!\!\!\!\!\overline{\mathscr{O}},\,~1 \leq i \leq k.
	\end{array}
	\right.
\end{equation}

	We notice that $\widehat\omvec(\mathbf{x},t_{2k-1/2})\equiv0$ in $\overline{\mathscr{O}}$.
	Therefore, $\widehat\omvec(\mathbf{x},t) \equiv 0$ in $\overline{\mathscr{O}}\times [t_{2k-1/2},1]$.
	Moreover,
\[
	\widehat\omvec(\xvec,t) =
	\sum\limits_{l=i}^{k}\omvec^{l}(\xvec,t)-\omvec^{i}(\xvec,t)~\hbox{ in }~\overline{\mathscr{O}}\times(t_{k+i-1/2},t_{k+i+1/2}),
	\ \ 1 \leq i \leq k-1.
\]

	Let $\omvec$ be the restriction of $\widehat\omvec$ to $\Ombar\times[0,1]$. 
	It belongs to $C^0([0,1];\Cvec^{1,\alpha}(\overline\Om;\mathbb{R}^3))$, satisfies
\[
	\left\{
	\begin{array}{lll} \dis
	{\omvec}_{t}+(\zvec\cdot\nabla){\omvec} =
	({\omvec}\cdot\nabla)\zvec - \overrightarrow{\textbf{k}} \times \nabla\theta &\hbox{ in }& \Om \times [0,1],
	\\ \dis
	\omvec(\xvec,0) = (\nabla \times \yvec_0)(\xvec) & \hbox{ in }& \Omega
	\end{array}
	\right.
\]
and, also, $\omvec(\mathbf{x},t) \equiv 0$~in~$\Ombar \times [t_{2k-1/2},1]$.

	Thanks to Lemma~\ref{lemma div-om}, $\omvec$ is divergence-free in $\Om\times(0,1)$.
	Consequently, from classical results, we know that there exists exactly one $\yvec$ in $C^0([0,1];\Cvec^{2,\alpha}(\overline\Om;\mathbb{R}^3))$ such that
\begin{equation}\label{yvecgorro}
\left\{
\begin{array}{lll}
	\nabla\times \yvec=\omvec , \ \ \nabla\cdot  \yvec=0   & \hbox{in}& \overline{\Omega}\times(0,1),\\
	\yvec\cdot\nvec = (\mu \yvec_0+\overline{\yvec})\cdot \nvec   & \hbox{on}& \Gamma\times(0,1).
\end{array}
\right.
\end{equation}
	Since $\yvec$ is uniquely determined by $\zvec$, we write $F(\zvec) = \yvec$.
	The mapping $F: \Rvec_\nu \mapsto \Rvec'$ is thus well defined.

	In view of some estimates similar to those in the two-dimensional case, we see that the initial data can be chosen small enough to have $F(\Rvec_\nu) \subset \Rvec_\nu$.
	More precisely, one has:

\begin{lemma}\label{Lemma-data-small}
	There exists $\delta>0$ such that, if $\{\|\yvec_0\|_{2,\alpha},\|\theta_0\|_{2,\alpha}\}\leq\delta$, one has $F(\zvec) \in \Rvec_\nu$ for all $\zvec \in \Rvec_\nu$.
\end{lemma}	

	The end of the proof of Proposition~\ref{P2} is very similar to the final part of Section~\ref{Sec4}.
	
	Essentially, what we have to prove is that, for some $m \geq 1$, $F^m$ is a contraction for the usual norm in~$C^0([0,1];\Cvec^{1,\alpha}(\overline\Om;\mathbb{R}^3))$.
	Indeed, after this, we can apply Theorem~\ref{fixedpoint} with $B_1 = C^0([0,1];\Cvec^{1,\alpha}(\overline{\Om};\mathbb{R}^3))$, $B_2 = C^0([0,1];\Cvec^{2,\alpha}(\overline{\Om};\mathbb{R}^3))$, $B = \Rvec_\nu$ and~$G = F$ and deduce the existence of a fixed-point of the extension $\widetilde F$ in the closure of $\Rvec_\nu$ in~$C^0([0,1];C^{1,\alpha}(\overline{\Om};\mathbb{R}^3))$.
	
	But this can be done easily, arguing as in the proof of Lemma~\ref{iteracion}.
	For brevity, we omit the details.
	
%%%%%%%%%%%%%%%%%%%%%%%%%%%%%%%%%%%%%%%%%%
%%%% SECTION 6
%%%%%%%%%%%%%%%%%%%%%%%%%%%%%%%%%%%%%%%%%%

\section{Proof of Theorem~\ref{T-HEAT-INV-BOUS}}\label{Sec6}

	Theorem~\ref{T-HEAT-INV-BOUS} is an easy consequence of the following result:

\begin{propo}\label{P4}
	For any $T^* > 0$ and any $\yvec_0\in \Cvec_0^{2,\alpha}(\Ombar ;\mathbb{R}^N)$, there exists $\eta>0$ such that, if~$\theta_0 \in C^{2,\alpha}(\Ombar)$, $\theta_0=0$ on~$\Gamma\backslash\gamma$ and~$\|\theta_0\|_{2,\alpha} \leq \eta$, the system
\begin{equation}\label{p-ext111}
	\left\{
\begin{array}{lll}
			\noalign{\smallskip} \dis \yvec_t + (\yvec \cdot \nabla) \yvec
			= -\nabla p + \mathbf{k} \, \theta  &\hbox{ in }& \Om\times(0,T^*),	\\
			\noalign{\smallskip} \dis  \nabla \cdot \yvec = 0  &\hbox{ in }& \Om\times(0,T^*),	\\
			\noalign{\smallskip} \dis \theta_t + \yvec \cdot \nabla \theta = \kappa\, \Delta \theta  &\hbox{ in }& \Om\times(0,T^*),	\\
			\noalign{\smallskip} \dis\yvec\cdot\nvec = 0																&\hbox{ on } &\Gamma\times(0,T^*),\\
			\noalign{\smallskip} \dis \theta=0																		&\hbox{ on } &(\Gamma\backslash\gamma)\times(0,T^*),\\
			\noalign{\smallskip} \dis \yvec(\xvec,0) = \yvec_0(\xvec),~\theta(\xvec,0) = \theta_0(\xvec)   &\hbox{ in } & \Om,			\\
	\end{array}
	\right.
\end{equation}
	possesses at least one solution $\yvec\in C^0([0,T^*];\Cvec^{2,\alpha}(\Ombar;\mathbb{R}^N))$,
	$\theta\in C^{0}([0,T^*]; C^{2,\alpha}(\Ombar ))$ and
	$p\in \mathcal{D'}(\Om\times (0,T^*))$ that satisfies
\begin{equation}\label{thetaT*}
\theta(\xvec,T^*)=0\quad \hbox{in} \quad\Omega.
\end{equation}
\end{propo}

	Indeed, if Proposition \ref{P4} holds, we can consider \eqref{inv-bous} and control first the temperature $\theta$ exactly to zero at a time $T^*<T$.
	To do this, we need initial data as above, that is, $\yvec_0 \in \Cvec_0^{2,\alpha}$ and~$\theta_0 \in C^{2,\alpha}(\Ombar)$ such that $\theta_0=0$ on~$\Gamma\backslash\gamma$ and~$\|\theta_0\|_{2,\alpha} \leq \eta$.
	Then, in a second step, we can apply the results in~\cite{Coron3} and~\cite{Glass3}  to the Euler system in~$\Om \times (T^*,T)$, with initial data $\yvec(\cdot\,,T^*)$. In other words, we can find new controls in $(T^*,T)$ that drive the velocity field exactly to any final state $\yvec_1$.

\

\noindent
{\bf Proof of Proposition~\ref{P4}:}
	For simplicity, we will consider only the case $N=2$.
	We will apply a fixed-point argument that guarantees the existence of a solution to~\eqref{p-ext111}-\eqref{thetaT*}.
	
	We start from an arbitrary $\overline\theta \in C^0([0,T^*];C^{1,\alpha}(\overline\Om))$.
	To this $\overline\theta$, arguing as in~Section~\ref{Sec3}, we can associate a field $\yvec\in C^0([0,T^*];\Cvec^{2,\alpha}(\Ombar;\mathbb{R}^N))$ verifying
\[
\left\{
\begin{array}{ll}
			\noalign{\smallskip} \dis \yvec_t + (\yvec \cdot \nabla) \yvec
			= -\nabla p + \mathbf{k}\,\overline\theta  				&\hbox{ in} \ \Om\times(0,T^*),	\\
			\noalign{\smallskip} \dis  \nabla \cdot \yvec = 0  				&\hbox{ in} \ \Om\times(0,T^*),	\\
			\noalign{\smallskip} \dis  \yvec\cdot\nvec = 0  				&\hbox{ on} \ \Gamma\times(0,T^*),	\\
			\noalign{\smallskip} \dis \yvec(\xvec,0) = \yvec_0(\xvec)		 &\hbox{ in} \ \Om			\\
	\end{array}
\right.
\]
	and
\[
	\|\yvec\|_{0,2,\alpha}\leq C(\|\yvec_0\|_{2,\alpha}+\|\overline\theta\|_{0,2,\alpha}).
\]

	Let $\tilde \Om \subset \mathbb{R}^2$ be a connected open set with boundary $\tilde\Gamma = \partial\tilde\Om$ of class $C^2$ such that $\Om\subset\tilde \Om$ and $\tilde\Gamma \cap \Gamma = \Gamma \setminus \gamma$
	(see~Fig.~\ref{fig1}).
	Let $\om\subset \tilde \Om\setminus\overline\Om$ be a non-empty open subset.
\begin{figure}[h]
\centering\includegraphics[scale=0.45]{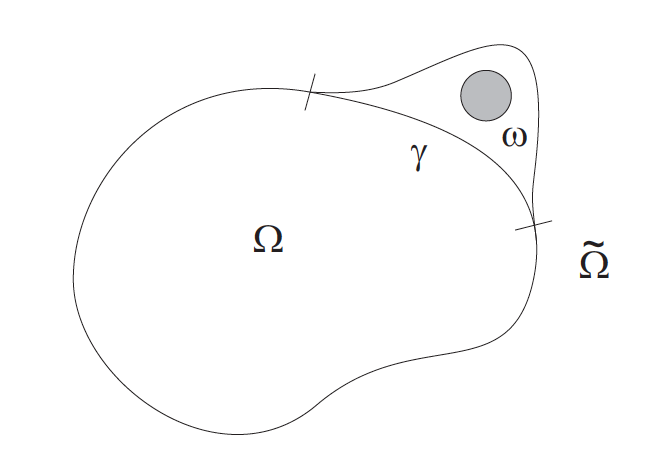}
\caption{The domain $\tilde\Om$ and the subdomain $\om$.}
\label{fig1}
\end{figure}

	Then, as in~Theorem~\ref{NC-Parab}, we associate to $\yvec$ a pair $(\tilde\theta,\tilde v)$ satisfying
\[
\left\{
\begin{array}{lll}
    \noalign{\smallskip}\dis \tilde\theta_t + \pi(\yvec)\cdot\nabla\tilde\theta = \kappa\Delta\tilde\theta + \tilde v 1_{\om}
    &\hbox{in} \ \tilde{\Om} \times (0,T^*),\\
    \noalign{\smallskip} \tilde\theta=0  &\hbox{on} \ \tilde{\Gamma} \times (0,T^*),\\
   \noalign{\smallskip} \tilde\theta(\xvec,0)=\tilde\pi(\theta_0)(\xvec),~~\tilde\theta(\xvec,T^*)=0 &\hbox{in} \ \tilde{\Om},
    \end{array}
\right.
\]
	where $\pi$ and $\tilde \pi$ are extension operators from~$\Om$ into~$\tilde \Om$ that preserve regularity.
	Let $\theta$ be the restriction of~$\tilde \theta$ to~$\overline\Om\times[0,T^*]$.
	Then, $\theta$ satisfies:
\[
\left\{
\begin{array}{ll}
    \noalign{\smallskip}\dis \theta_t + \yvec\cdot\nabla\theta = \kappa\Delta\theta 	&\hbox{in} \ \Om\times(0,T^*),\\
    \noalign{\smallskip} \theta=\tilde\theta 1_{\gamma}             									&\hbox{on} \ \Gamma\times(0,T^*),\\
   \noalign{\smallskip} \theta(\xvec,0)=\theta_0(\xvec),~~\theta(\xvec,T^*)=0		&\hbox{in} \ \Om.
    \end{array}
\right.
\]

%\[
%	\|\theta(\cdot\,,t)\|_{0,\alpha'}\leq \|\theta_0\|_{2,\alpha'}e^{C(\|\yvec_0\|_{2,\alpha}+\|\tilde\theta\|_{0,2,\alpha})}
%	\quad \forall t \in [0,T^*],
%\]
	
	Moreover, from parabolic regularity theory, it is clear that the following inequalities hold:
\[
	\|\theta_t\|_{0,0,\alpha}+\|\theta\|_{0,2,\alpha} \leq C \|\theta_0\|^2_{{2,\alpha}} \, e^{C \|\yvec\|_{0,2,\alpha}}
	\leq C \|\theta_0\|_{{2,\alpha}} \, e^{C(\|\yvec_0\|_{2,\alpha}+\|\overline\theta\|_{0,2,\alpha})}.
\]

	Now, let us introduce the Banach space
\[
	W = \{\, \theta\in C^0([0,T^*]; C^{2,\alpha}(\overline\Om)) : \theta_t\in C^0([0,T^*];C^{0,\alpha}(\overline\Om)) \,\}
\]
	and let us consider the closed ball
$$
	B := \{\, \overline\theta \in  C^0([0,T^*]; C^{1,\alpha}(\overline\Om)) : \|\overline\theta\|_{0,1,\alpha} \leq 1 \,\}
$$
	and the mapping $\Lambda$, with
\[
	\Lambda(\overline\theta) = \theta \quad \forall \overline\theta \in C^0([0,T^*]; C^{1,\alpha}(\overline\Om)).
\]

	Obviously, $\Lambda$ is well defined.
	Furthermore, in view of the previous inequalities, it maps continuously the whole space $C^0([0,T^*]; C^{1,\alpha}(\overline\Om))$ into $W$, that is compactly embedded in~$ C^0([0,T^*]; C^{1,\alpha}(\overline\Om))$, in view of the classical results of the Aubin-Lions kind, see for instance~\cite{Simon}.

	On the other hand, if $\eta > 0$ is sufficiently small
	(depending on $\|\yvec_0\|_{2,\alpha}$) and $\|\theta_0\|_{2,\alpha}\leq \eta$, $\Lambda$ maps $B$ into itself.
	Consequently, the hypotheses of Schauder's Theorem are satisfied and $\Lambda$ possesses at least one fixed-point in $B$.
\Fin

%\noindent
{\bf Acknowledgements:}
	The authors are grateful to the anonymous referees for their valuable comments and suggestions, that have allowed to improve a previous version of the paper.

%%%%%%%%%%%%%%%%%%%%%%%  COMMENTS
%
%
%   SIMPLY CONNECTED
%
%   REGULARITY
%
%   GLOBAL RESULT K>0
%
%

%%%%%%%%%%%%%%%%%%%%%%%%%%%%%%%%%%%%%%%%%%
%%%% REFERENCES
%%%%%%%%%%%%%%%%%%%%%%%%%%%%%%%%%%%%%%%%%%

{%\scriptsize

%\bibliographystyle{siam}
%\bibliography{biblio}
%\end{document}

}

\end{document}